\documentclass[12pt]{amsart}
\usepackage{geometry} 
\usepackage{amsmath,amsthm,amsfonts,amssymb}
\usepackage{graphics,xy}
\usepackage{url}
\usepackage{xcolor}
\usepackage{graphicx}
\usepackage{ulem}
\usepackage{svg}
\usepackage{tikz}
\usepackage{hyperref}
\usepackage[utf8]{inputenc}
\usepackage{subfigure}

\newtheorem{lemma}{Lemma}[section]
\newtheorem{conjecture}{Conjecture}[section]
\newtheorem{proposition}[lemma]{Proposition}
\newtheorem{theorem}[lemma]{Theorem}

\newtheorem{corollary}[lemma]{Corollary}

\theoremstyle{definition}
\newtheorem{definition}[lemma]{Definition}
\newtheorem{remark}[lemma]{Remark}

\newcommand{\Z}{\mathbb{Z}}
\newcommand{\R}{\mathbb{R}}
\newcommand{\Q}{\mathbb{Q}}
\newcommand{\C}{\mathbb{C}}

\newcommand{\Sk}{\mathcal{S}}
\newcommand{\ra}{\longrightarrow}
\newcommand{\Tr}{\mathrm{Tr}}
\newcommand{\F}{F_{1,2}\simtimes S^1}

\newcommand{\slC}{\mathrm{SL}_2(\mathbb{C})}
\newcommand{\rp}{\mathbb{RP}^3}
\DeclareMathOperator{\tr}{Tr}
\newtheorem*{futuretheoreminner}{Theorem~\thefuturetheoreminner}
\newcommand{\thefuturetheoreminner}{} 
\newcommand\simtimes{\mathbin{%
		\stackrel{\sim}{\smash{\times}\rule{0pt}{0.9ex}}%
}}
\newcommand\maybesimtimes{\mathbin{%
		\stackrel{\left(\sim\right)}{\smash{\times}\rule{0pt}{0.9ex}}%
}}
\ExplSyntaxOn

\prop_new:N \g_alevel_future_prop

\NewDocumentEnvironment{futuretheorem}{ m o +b }
{
	\renewcommand{\thefuturetheoreminner}{\ref{#1}}
	\IfNoValueTF{#2}
	{\futuretheoreminner}
	{\futuretheoreminner[#2]}
	#3
}
{
	\endfuturetheoreminner
	\prop_gput:Nnn \g_alevel_future_prop { #1 } { #3 }
	\IfValueT{#2}{ \prop_gput:Nnn \g_alevel_future_prop { #1-attr } { #2 } }
}

\NewDocumentCommand{\pasttheorem}{m}
{
	\prop_if_in:NnTF \g_alevel_future_prop { #1-attr }
	{
		\begin{theorem}[\prop_item:Nn \g_alevel_future_prop { #1-attr }]
		}
		{
			\begin{theorem}
			}
			\label{#1}
			\prop_item:Nn \g_alevel_future_prop { #1 }
		\end{theorem}
	}
	
	\ExplSyntaxOff

\title[On torsion in Kauffman bracket skein modules]{On torsion in the Kauffman bracket skein module of $3$-manifolds}

\author{Giulio Belletti}
\address{Institut de Mathématiques de Bourgogne, UMR 5584 CNRS, Université Bourgogne Franche-Comté, F-2100 Dijon, France}
\email{gbelletti451@gmail.com}
\author{Renaud Detcherry}
\date{} 

\address{Institut de Mathématiques de Bourgogne, UMR 5584 CNRS, Université Bourgogne Franche-Comté, F-2100 Dijon, France}
\email{renaud.detcherry@u-bourgogne.fr}

\begin{document}

\begin{abstract}We study Kirby problems 1.92(E)-(G) of \cite{Kirby}, which, roughly speaking, ask for which compact oriented $3$-manifold $M$ the Kauffman bracket skein module $\Sk(M)$ has torsion as a $\Z[A^{\pm 1}]$-module. We give new criteria for the presence of torsion in terms of how large the $\slC$-character variety of $M$ is. This gives many counterexamples to question 1.92(G)-(i) in Kirby's list. For manifolds with incompressible tori, we give new effective criteria for the presence of torsion, revisiting the work of Przytycki and Veve \cite{Prz99}\cite{Veve}. We also show that $\mathcal{S}(\R P^3\# L(p,1))$ has torsion when $p$ is even. Finally, we show that for $M$ an oriented Seifert manifold, closed or with boundary, $\mathcal{S}(M)$ has torsion if and only if $M$ admits a $2$-sided non-boundary parallel essential surface.
\end{abstract}
\maketitle
\section{Introduction}
\label{sec:intro}  
In this paper we explore the problem of finding torsion in the Kauffman bracket skein module of a $3$-manifold, whose definition, given below, was originally introduced independently by Przytycki \cite{Przytycki} and Turaev \cite{Turaev}:

\begin{definition}\label{def:KBSM} For $M$ a compact oriented $3$-manifold, the Kauffman bracket skein of $M$, denoted with $\Sk(M)$, is the quotient of the free $\Z[A^{\pm 1}]$-module generated by isotopy classes of framed links in $M,$ by the Kauffman relations, which are the following relations between framed links that are identical in the complement of a ball:
\begin{center}
	\def \svgwidth{1.1\columnwidth}
\begingroup%
  \makeatletter%
  \providecommand\color[2][]{%
    \errmessage{(Inkscape) Color is used for the text in Inkscape, but the package 'color.sty' is not loaded}%
    \renewcommand\color[2][]{}%
  }%
  \providecommand\transparent[1]{%
    \errmessage{(Inkscape) Transparency is used (non-zero) for the text in Inkscape, but the package 'transparent.sty' is not loaded}%
    \renewcommand\transparent[1]{}%
  }%
  \providecommand\rotatebox[2]{#2}%
  \newcommand*\fsize{\dimexpr\f@size pt\relax}%
  \newcommand*\lineheight[1]{\fontsize{\fsize}{#1\fsize}\selectfont}%
  \ifx\svgwidth\undefined%
    \setlength{\unitlength}{310.27629089bp}%
    \ifx\svgscale\undefined%
      \relax%
    \else%
      \setlength{\unitlength}{\unitlength * \real{\svgscale}}%
    \fi%
  \else%
    \setlength{\unitlength}{\svgwidth}%
  \fi%
  \global\let\svgwidth\undefined%
  \global\let\svgscale\undefined%
  \makeatother%
  \begin{picture}(1,0.08365284)%
    \lineheight{1}%
    \setlength\tabcolsep{0pt}%
    \put(0,0){\includegraphics[width=\unitlength,page=1]{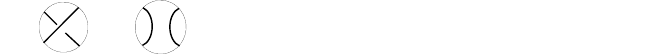}}%
    \put(0.1543921,0.02370045){\color[rgb]{0,0,0}\makebox(0,0)[lt]{\lineheight{0}\smash{\begin{tabular}[t]{l}$=A$\end{tabular}}}}%
    \put(0.29591666,0.02367871){\color[rgb]{0,0,0}\makebox(0,0)[lt]{\lineheight{0}\smash{\begin{tabular}[t]{l}$+A^{-1}$\end{tabular}}}}%
    \put(0.57577529,-0.02838106){\color[rgb]{0,0,0}\makebox(0,0)[lt]{\lineheight{0}\smash{\begin{tabular}[t]{l} \end{tabular}}}}%
    \put(0.53325876,0.02471831){\color[rgb]{0,0,0}\makebox(0,0)[lt]{\lineheight{0}\smash{\begin{tabular}[t]{l}$L \ \coprod$ \end{tabular}}}}%
    \put(0,0){\includegraphics[width=\unitlength,page=2]{kauffman.pdf}}%
    \put(0.67966381,0.02531714){\color[rgb]{0,0,0}\makebox(0,0)[lt]{\lineheight{0}\smash{\begin{tabular}[t]{l}$=(-A^2-A^{-2}) L$\end{tabular}}}}%
    \put(-0.00304793,0.0192645){\color[rgb]{0,0,0}\makebox(0,0)[lt]{\lineheight{0}\smash{\begin{tabular}[t]{l}K1:\end{tabular}}}}%
    \put(0.47835686,0.02531722){\color[rgb]{0,0,0}\makebox(0,0)[lt]{\lineheight{0}\smash{\begin{tabular}[t]{l}K2:\end{tabular}}}}%
    \put(0,0){\includegraphics[width=\unitlength,page=3]{kauffman.pdf}}%
  \end{picture}%
\endgroup%

\end{center}
\end{definition}

For some applications in this paper, we will also consider the Kauffman bracket skein module over $\Q(A)$, rather than $\Z[A,A^{-1}]$; in this case, we denote it with $\Sk(M,\Q(A))$. Moreover, for $\zeta\in \C^*,$ we also introduce the notation
$$S_{\zeta}(M):=S(M)\underset{A=\zeta}{\otimes}\C.$$

We will refer to the Kauffman bracket skein module of $M$ as simply the skein module of $M,$ since we will not consider any other kind of skein modules in this paper. 
For $M$ a $3$-manifold, we will also introduce its $\slC$-character scheme $\mathcal{X}(M)$ as
$$\mathcal{X}(M)=\mathrm{Hom}(\pi_1(M),\slC)//\slC,$$
where $//$ is the GIT quotient. We will denote by $X(M)$ the $\slC$-character variety of $M,$ which is the underlying affine algebraic set to $\mathcal{X}(M)$. One of the motivations of studying skein modules of $3$-manifolds is given by the following theorem, which states that skein modules are deformations by quantization of the $\slC$-character variety of $M.$ 
Recall that for $x$ an unoriented loop in $M,$ there is a well defined element $t_x\in \C[\mathcal{X}(M)],$ such that for any $\rho \in \mathrm{Hom}(\pi_1(M),\slC),$ we have $t_x(\rho)=-\mathrm{Tr}(\rho(x)).$
\begin{theorem}
	\cite{Bullock}\cite{PS00} There is a well-defined isomorphism of algebras $\Psi:S_{-1}(M)\longrightarrow \C[\mathcal{X}(M)]$ such that for any link $L=L_1\cup \ldots \cup L_n$ in $M,$ we have
	$$\Psi(L)=t_{L_1}\ldots t_{L_n}.$$
\end{theorem}

Another motivating question for the study of skein modules of $3$-manifolds is to understand their relationship with the presence of $2$-sided essential surfaces in $M.$ Indeed, the following conjecture, which is essentially a conjecture of Przytycki that first appeared as Problem 1.92(E) in \cite{Kirby}, predicts that torsion in the skein module can detect incompressible surfaces: 
\begin{conjecture}\label{conj:torsion}\cite{Kirby}\cite{DKS2} Let $M$ be a compact oriented $3$-manifold. Then the following are equivalent:
	\begin{itemize}
		\item[(1)] every two-sided, closed, essential surface in $M$ is parallel to the boundary.
		\item[(2)]$\mathcal{S}(M)$ is torsion-free.
	\end{itemize}
	
\end{conjecture} 

We note that, strictly speaking, Przytycki only conjectured $(2)\implies (1),$ while the other implication was recently conjectured by Kalfagianni, Sikora and the second author \cite{DKS2}; furthermore they conjecture that, if $M$ is closed, then both are equivalent to
\begin{itemize}
	\item[(3)]$\mathcal{S}(M)$ is finitely generated over $\Z[A^{\pm 1}].$
\end{itemize}

We will be concerned solely with the implication $(2)\implies (1)$; in other words, we will be concerned with showing that a two-sided, closed, non boundary parallel essential surface gives torsion in $\Sk(M)$.

The examples in the literature thus far mostly concern the case of low-genus surfaces. For genus $0$ an almost complete solution is known.

\begin{theorem}\cite[Theorem 1.7]{GSZ}\label{thm:GSZ}
	Let $M_1,M_2$ be two compact orientable manifolds, neither of which is homeomorphic to $\#^r \rp$ minus some balls (for $r\geqslant 0)$. Then $\Sk(M_1\# M_2)$ has torsion.
\end{theorem}

Note that when $M$ has a non-boundary parallel essential sphere, either $M$ is a connected sum or $M=S^1\times S^2$; in the latter case, $\Sk(S^1\times S^2)$ is known explicitly and has torsion (see \cite{HP95}).

One of the results of this paper is the following:
\begin{theorem}\label{thm:rp3}
	Take $r\geq 1$ and $p$ even; then $\Sk(L(p,1)\#^r \rp)$ has torsion.
\end{theorem}

Let us remark that those connected sums have finite $\slC$-character variety. The approach in \cite{GSZ} to prove Theorem \ref{thm:GSZ} can be used to find torsion in the skein module of these $3$-manifolds, as they used a criterion of Przytycki (\cite[Thm 4.2(b)]{Prz99}) which fails for these manifolds. To find torsion we introduce a new technique, which involves studying the skein module at a $4$th root of unity.
\\
\\
The first example of an irreducible manifold with torsion in its skein module was the double of the Figure Eight knot (as shown by \cite{Veve}); in this case the torsion arises due to an essential torus. 
In this paper we provide effective criteria, based on representation theory, to find torsion in skein modules arising from essential tori; in particular we are able to use them to prove the following theorem, settling the implication $(2)\implies(1)$ of Conjecture \ref{conj:torsion} for Seifert manifolds.

\begin{theorem}\label{thm:SFStorsion}
Suppose $M$ is an orientable Seifert manifold that contains an incompressible, non-boundary parallel, closed orientable surface. Then $\Sk(M)$ has $(A\pm 1)$-torsion.
\end{theorem}

The criteria used in the proof of Theorem \ref{thm:SFStorsion} take several forms. As a first ingredient, we give effective criteria, depending on the $\slC$-character variety of $M,$ under which the presence of a separating torus in $M$ (Theorem \ref{thm:torsiontorus}), or of a non-separating torus in $M$ (Theorem \ref{thm:nsep-torus}), produces torsion in $\Sk(M).$ Those criteria are effective in the sense that one can work out explicit torsion elements in $\Sk(M)$ from them. We note that those criteria are a generalization of some results of Przytycki \cite{Prz99} and Veve \cite{Veve}.

Moreover, in some more complicated cases, we rely on some non-effective criteria (Theorem \ref{thm:infiniteX(M)}, Corollary \ref{cor:infiniteX(M)} and Theorem \ref{thm:largeX(M)_nclosed}), which morally speaking show that if the character variety $X(M)$ is "large" then $\Sk(M)$ has torsion. The first of these criteria concerns closed manifolds, and can be thought of as a consequence of the finiteness theorem for skein modules, a celebrated result of Gunningham, Jordan and Safronov which states the following:
\begin{theorem}
	\label{thm:finiteness}\cite{GJS19} For any closed compact oriented $3$-manifold $M,$ the skein module $\Sk(M,\Q(A))$ is a finite dimensional $\Q(A)$-vector space.
\end{theorem}
We note that a version of this theorem for manifolds with boundary has been conjectured to hold by the second author (\cite[Conjecture 3.3]{Det21}, or Conjecture \ref{conj:finiteness_boundary} below), and is known as the strong finiteness conjecture for manifolds with boundary.

\begin{conjecture}(Strong finiteness conjecture for manifolds with boundary \cite{Det21})
	\label{conj:finiteness_boundary} Let $M$ be a compact oriented $3$-manifold. Then there is a family $\mathcal{C}=\lbrace c_1,\ldots, c_n\rbrace$ of essential, disjoint, pairwise non parallel simple closed curves on $\partial M,$ such that $\mathcal{S}(M,\Q(A))$ is a finitely generated $\Q(A)[c_1,\ldots,c_n]$-module.
\end{conjecture}

We remark that strictly speaking, this conjecture is stronger than Conjecture 3.3 in \cite{Det21}. Note that for disjoint simple closed curves on $\partial M,$ the skein module $\mathcal{S}(M,\Q(A))$ can be considered a $\Q(A)[c_1,\ldots,c_k]$-module, since as elements of $\Sk(\partial M)$ the curves $c_1,\ldots,c_k$ commute, and $\mathcal{S}(M)$ has a natural structure of $\Sk(\partial M)$-module.
 
Our second criterion, Theorem \ref{thm:largeX(M)_nclosed}, is a variant of Corollary \ref{cor:infiniteX(M)} that applies to $3$-manifolds with boundary for which the strong finiteness conjecture is known. 

In order to apply Theorem \ref{thm:largeX(M)_nclosed} to the proof of Theorem \ref{thm:SFStorsion}, we will prove the strong finiteness conjecture for certain Seifert manifolds with boundary:
\begin{theorem}
	\label{thm:strongFinitenessMobius} Let $F_{1,1}$ be the M\"obius band and $F_{1,2}$ the M\"obius band with a disk removed. Then the strong finiteness conjecture holds for $\F$ and for any Seifert manifold over $F_{1,1}$ with one singular fiber.
\end{theorem}
This extends results of Aranda and Ferguson \cite[Theorem 4.2]{AF22}, who proved the strong finiteness conjecture for Seifert manifolds with one boundary component and exceptional fibers of parameters $(1,p_i)$.

\textbf{Structure of the paper:} In Section \ref{sec:largeX(M)} we provide the non-effective criteria based on the largeness of $X(M)$. In Section \ref{sec:largeXapplications} we apply these criteria to obtain some examples of manifolds that have torsion in the skein module but only have incompressible surfaces of large genus. In Section \ref{sec:tori} we provide criteria to detect torsion arising from incompressible tori, both separating and non-separating. In Section \ref{sec:finiteness} we prove the strong finiteness conjecture for some Seifert manifolds with non-orientable base, in order to apply to them the criteria from Section \ref{sec:largeX(M)}. In Section \ref{sec:Seifert} we combine all previous results to prove Theorem \ref{thm:SFStorsion}. Finally in Section \ref{sec:rp3} we provide the proof of Theorem \ref{thm:rp3}, finding torsion in manifolds of the form $L(p,1)\#^r\rp$ for $p$ even and $r\geq 0$.

\textbf{Acknowledgements.} Over the course of this work, both authors were partially supported by by the ANR project "NAQI-34T" (ANR-23-ERCS-0008) and by the project "CLICQ" of the Région Bourgogne-Franche Comté. The IMB, host institution of the authors, receives support from the EIPHI Graduate School (contract ANR-17-EURE-0002).

\section{Torsion in the skein modules of manifolds with large character variety}
\label{sec:largeX(M)}
In this section, we will prove two results indicating that a compact oriented $3$-manifold $M$ with a large $\slC$-character variety $X(M)$ has torsion in its skein module.
Our first result concerns closed $3$-manifolds, and is a consequence of the finiteness theorem for skein modules. 
\begin{theorem}\label{thm:infiniteX(M)}
	Let $M$ be a closed compact oriented $3$-manifold and let $\zeta\in \C^*.$ If $\dim_{\C}(S_{\zeta}(M))>\dim_{\Q(A)}(S(M,\Q(A)))$ then $\mathcal{S}(M)$ has $A-\zeta$-torsion.
\end{theorem}
\begin{proof}

Suppose $\dim_{\C}\left(S(M,\Q(A))\right)=n-1$, take $\gamma_1,\dots, \gamma_n$ elements of $\mathcal{S}(M)$ such that their images in $S_\zeta(M)$ are linearly independent, and denote with $\tilde{\gamma}_1,\dots, \tilde{\gamma}_n$ their image in $S(M,\mathbb{Q}(A))$.

Because the latter is $n-1$-dimensional, there must be a linear relation between them; up to reordering, we can assume that it is $\sum_{i=1}^N P_i(A) \tilde{\gamma}_i=0$, with each $P_i(A)\in \Q(A)$ different from $0$. After multiplying the $P_i$s by the least common multiple of their denominators, we can assume that each $P_i(A)$ is in $ \mathbb{Z}[A,A^{-1}]$. 
This implies that $\sum_{i=1}^N P_i(A) \gamma_i$ is also equal to $0$ as an element of $\mathcal{S}(M)$. When evaluated at $\zeta$, this sum becomes $\sum_{i=1}^N P_i(\zeta)\gamma_i$, and because the $\gamma_i$s are linearly independent in $S_\zeta(M)$,  we have that for $i=1,\ldots, N, \ P_i(\zeta)=0.$  Therefore, $A-\zeta$ divides all $P_i$s and we can write the sum as $(A-\zeta)^k\left(\sum_i^N Q_i(A)\gamma_i\right)=0$, where $Q_i(\zeta)\neq 0$ for at least one $i$. Additionally we can see that $\sum_i^N Q_i(A)\gamma_i\neq 0$ in $\mathcal{S}(M)$, because otherwise it would also be equal to $0$ in $S_\zeta(M)$, which would give a non-trivial linear relation between the $\gamma_i$s in $S_\zeta(M)$.
\end{proof}
\begin{corollary}\label{cor:infiniteX(M)}
Let $M$ be a closed oriented $3$-manifold such that $X(M)$ is positive dimensional. Then $\mathcal{S}(M)$ has $A-\zeta$ torsion, for any $\zeta$ root of unity of odd order or order $\equiv 2 \mod 4.$
\end{corollary}
\begin{proof}
When $X(M)$ is positive dimensional, the dimension of $\C[X(M)]$ must be infinite; since this is isomorphic to $S_{\pm 1}(M)$, it implies that the latter must also be infinite dimensional. Moreover, by \cite[Theorem 2.1]{DKS}, we have $\dim S_{\zeta}(M)\geq \dim S_{-1}(M)$ for any $\zeta$ of order $2 \mod4,$ and similarly for odd order roots, since $S_{x}(M)$ and $S_{-x}(M)$ are isomorphic for any $x\in \C$ by Barett \cite{Barrett}. Then the result is a consequence of Theorem \ref{thm:infiniteX(M)}.
\end{proof}
While the criterion given by Theorem \ref{thm:infiniteX(M)} is, as we shall see, powerful, it has two drawbacks. First, as it relies on the finiteness theorem, it does not apply to the case of manifolds with boundary. Moreover, it does not provide explicit torsion elements. While most of the applications concern the case $\zeta=\pm 1,$ we will see that Theorem \ref{thm:infiniteX(M)} can be used for other choices of $\zeta$ to detect torsion even when the character variety is finite.
	
Our next result is a refinement of Theorem \ref{thm:infiniteX(M)}, that is applicable to manifolds with boundary.

\begin{theorem}\label{thm:largeX(M)_nclosed}
	Let $M$ be a compact oriented $3$-manifold with boundary, and let $c_1,\ldots,c_k$ be disjoint simple closed curves on $\partial M.$ Assume that $\mathcal{S}(M,\Q(A))$ is finitely generated over $\Q(A)[c_1,\ldots,c_k],$ that the trace functions $t_{c_1},\ldots ,t_{c_k}$ are algebraically independent in $\C[X(M)],$  and that $\dim X(M)>k.$ 
	
	Then $\Sk(M)$ admits $(A\pm 1)$-torsion.
\end{theorem}
Compared to Theorem \ref{thm:infiniteX(M)}, Theorem \ref{thm:largeX(M)_nclosed} is trickier to apply; for one, Conjecture \ref{conj:finiteness_boundary} is still open. Furthermore, even if Conjecture \ref{conj:finiteness_boundary} were to be proven, to apply the criterion one would still need to prove algebraic independence of the trace functions. Nonetheless, we will see in Section \ref{sec:Seifert} an example of application of this theorem.
\begin{proof}
 Recall that $\C[X(M)]$ is generated as a ring by trace functions $t_{\gamma}$ where $\gamma \in \pi_1(M).$ If $\dim X(M)>k$ then there must be $\gamma\in \pi_1(M)$ such that $t_{\gamma}$ is transcendental over $\C[t_{c_1},\ldots,t_{c_k}].$ Consider a framed knot $K$ that represents the conjugacy class of $\gamma$ (or $\gamma^{-1}$) in $\pi_1(M).$ Let $K^n$ denote the link obtained by taking $n>0$ parallel copies of $K.$
 
 Since $\Sk(M)$ is finitely generated over $\Q(A)[c_1,\ldots,c_k],$ there exists $n>0$ such that the elements 
 $$\lbrace c_1^{n_1}\ldots c_k^{n_k} K^{n_0}  \ | \ \forall i, 0\leq n_i\leq n\rbrace$$ 
 are linearly dependent in $\Sk(M,\Q(A)).$ As the same elements, considered as elements of $\Sk_{-1}(M)$ are linearly independent by hypothesis, we can apply the same argument as in the proof of Theorem \ref{thm:infiniteX(M)} (clearing out denominators then factoring out the common $(A+1)$-factors) to conclude that $\Sk(M)$ has $(A+1)$-torsion. Then $\Sk(M)$ also has $(A-1)$-torsion by Barrett's isomorphism \cite{Barrett}. 
\end{proof}
	\section{Manifolds with positive dimensional character variety}\label{sec:largeXapplications}
	In this section, we give some examples of applications of Theorem \ref{thm:infiniteX(M)}.
	\subsection{Closed manifolds with torsion in \texorpdfstring{$\mathcal{S}(M)$}{S(M)}}
	\label{sec:infiniteX(M)}
	Using Theorem \ref{thm:infiniteX(M)}, we get the following corollary, which highlights that torsion in $\mathcal{S}(M)$ can be caused not only by spheres or tori, but also by higher genus surfaces:
	\begin{corollary}\label{cor:higherGenusSurf}
		For any closed hyperbolic $3$-manifold $M$ with $b_1(M)>0,$ there is $(A+1)$-torsion in $\mathcal{S}(M).$
		
		Moreover, for any $k\geq 0,$ there exists a closed compact oriented $3$-manifold $M$ such that $\mathcal{S}(M)$ has $(A+1)$-torsion but $M$ does not contain any incompressible surface of genus $\leq k.$
	\end{corollary}
	To the authors' knowledge, those are the first examples in the literature of closed $3$-manifolds without spheres or tori such that $\mathcal{S}(M)$ has torsion.
	\begin{proof}
		The first claim is a direct consequence of Theorem \ref{thm:infiniteX(M)}, as $b_1(M)>0$ implies that $X(M)$ is positive dimensional: indeed if $b_1(M)>0$ then there is curve of abelian $\slC$-characters of $M.$
		
		Similarly, to prove the second claim, notice that it is sufficient to show that there exists a closed $3$-manifold $M$ such that $b_1(M)>0$ and $M$ does not contain any incompressible surface of genus $g\leq k.$ We prove the existence of such manifolds in the next section, as Proposition \ref{prop:curve_complex_construction}. 
	\end{proof}

	\subsection{Construction of manifolds with \texorpdfstring{$b_1(M)>0$}{b_1(M)>0} and no incompressible surfaces of low genus}

	In this section, which is independent of the rest of the paper, we will prove
	\begin{proposition}\label{prop:curve_complex_construction}
		For any $k,l\geq 0,$ there exists a closed $3$-manifold such that $b_1(M)>l$ and $M$ does not contain any incompressible surface of genus $g\leq k.$
	\end{proposition}

The proof will be based on curve complex techniques, in particular, the notion of Heegaard splitting distance, introduced by Hempel \cite{Hempel}, which we briefly recall below. For $\Sigma$ a closed connected oriented surface of genus at least $2,$ recall that the curve graph $\mathcal{C}(\Sigma)$ is the graph whose vertices are isotopy classes of simple closed curves on $\Sigma,$ and edges join vertices corresponding to two isotopy classes if and only if they have disjoint representatives. We write $d_{\mathcal{C}(\Sigma)}$ for the distance on the curve graph.

We recall that Thurston constructed a compactification of $\mathcal{C}(\Sigma)$ which is the set $PL(\Sigma)$ of projective measured laminations of $\Sigma,$ see \cite{FLP}.

For $M$ a $3$-manifold, a Heegaard splitting is a decomposition $M=H\underset{\Sigma}{\bigcup} H'$ where $\Sigma$ is a closed connected oriented surface and $H,H'$ are handlebodies. Its genus is the genus of $\Sigma.$ Hempel introduced a notion of distance of Heegaard splittings:

\begin{definition}\cite{Hempel}\label{def:splitting_dist}
	\label{def:splitting_distance}Let $M=H\underset{\Sigma}{\bigcup} H'$ be a Heegaard splitting of $3$-manifold $M$ of genus $g\geq 2.$ The splitting distance is $d_{\mathcal{C}(\Sigma)}(C,C'),$ where $C$ is the set of curves on $\Sigma$ that bound disks in $H$ and $C'$ is the set of curves on $\Sigma$ that bound disks in $H'.$
\end{definition}
Hempel proved that the set of splitting distances is unbounded for any surface of genus $g\geq 2.$
More precisely, he proved:
\begin{theorem}\label{thm:Hempel}
	\cite[Theorem 2.7]{Hempel} Let $H$ be a handlebody of genus at least $2$ with boundary $\Sigma,$ and $C$ the associated subset of $\mathcal{C}(\Sigma).$ If $f$ is a pseudo-Anosov element that satisfies the condition 
		\begin{center}(*) The closure in the set of projective measured laminations $PL(\Sigma)$ of $C$ does not contain the stable lamination of $f.$
	\end{center}
then the distance of the splitting $M=H\underset{f^n}{\bigcup}H$ tends to infinity.
\end{theorem} 
\begin{proof}[Proof of Proposition \ref{prop:curve_complex_construction}]
 We construct $M_n$ as the manifold with Heegaard splitting $M_n=H\underset{f^n}{\bigcup}H,$ where $H$ is a handlebody of boundary $\Sigma,$ a surface of genus $g\geq l\geq 2,$ and $f$ is a pseudo-Anosov element of the Torelli group $T(\Sigma)$ of $\Sigma,$ that satisfies the additional condition (*) of Theorem \ref{thm:Hempel}.
	
	As noted in the proof of \cite[Theorem 2.7]{Hempel}, by \cite{Masur}, the closure of $C$ is nowhere dense in $PL(\Sigma),$ while the set of stable laminations of Torelli pseudo-Anosov is stable under the action of $MCG(\Sigma),$ and hence dense in $PL(\Sigma).$ Indeed, by a result of Thurston \cite[Theorem 6.1]{FLP}, the action of $MCG(\Sigma)$ on $PL(\Sigma)$ is minimal.
	
	Therefore, there exists a pseudo-Anosov element $f$ in $T(\Sigma)$ that satisfies (*), by considering a conjugate of a given pseudo-Anosov element in $T(\Sigma)$ by a suitable element of $MCG(\Sigma)$.
	
	Note that since $f^n$ is in the Torelli group of $\Sigma,$ we have $H_1(M_n,\Z)=\Z^g.$
	Moreover, by Hempel \cite[Theorem 2.7]{Hempel}, the set of distances of the Heegaard splittings $H\underset{f^n}{\bigcup}H$ tend to $+\infty$ when $n$ tends to $+\infty.$
	
	However, by Hartshorn \cite{Hartshorn}, if the distance of a splitting of $M$ is larger than $2k,$ then $M$ does not contain any incompressible surface of genus $\leq k.$
	
	Therefore, for large $n,$ the manifold $M_n$ is a $3$-manifold without incompressible surfaces of genus $\leq k$ but $b_1(M_n)=g\geq l.$
	
\end{proof}
	
	\section{Torsion in manifolds with incompressible tori}\label{sec:tori}
	The goal of this section is to explain how torsion in the skein module $\mathcal{S}(M)$ of a compact oriented $3$-manifold $M$ can be caused by the presence of incompressible non-boundary parallel tori in $M.$ We revisit and expand old classical work of Przytycki \cite{Prz99} and Veve \cite{Veve}.
	Note that contrary to Section \ref{sec:largeX(M)}, the criteria we give in this section will be effective, giving explicit torsion in $\mathcal{S}(M),$ and independent of whether $M$ is closed or with boundary.
	\label{sec:torsion_tori}
	\subsection{The case of separating tori}
	\label{sec:torsion_septori}
	Throughout this section, let $M$ be a compact connected oriented $3$-manifold that contains an incompressible separating torus $T,$ and let $M_1$ and $M_2$ be the connected components of $M\setminus T.$

\begin{theorem}\label{thm:torsiontorus}
Let $M_1,M_2, M,T$ be as above. Suppose that $\rho:\pi_1(M)\ra SL(2,\C)$ is a representation satisfying the following:

\begin{itemize}
\item $\rho$ is irreducible;
\item $\rho$ restricts to non-abelian representations of $\pi_1(M_1)$ and $\pi_1(M_2)$;
\item $\rho$ restricts to a non-central representation of $\pi_1(T)\subseteq \pi_1(M)$.
\end{itemize}

Then $\mathcal{S}(M)$ contains $(A\pm 1)$-torsion.
\end{theorem}

The first step to prove this theorem is to obtain a torsion candidate; later we will prove that this element is actually non-zero.

\begin{figure}
    \centering
      \begin{minipage}{.45\textwidth}
 \centering   

\includegraphics[scale=0.45]{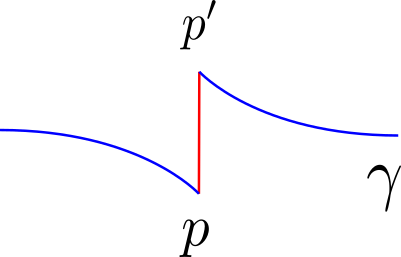}

   \end{minipage}   
  \begin{minipage}{.45\textwidth}
 \centering    
\includegraphics[scale=0.45]{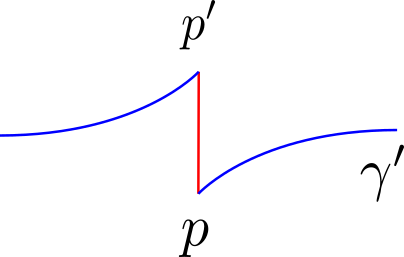}

   \end{minipage}
\caption{The curves $\gamma$ and $\gamma'$.}\label{fig:gamma}
\end{figure}
Fix a small segment in $T$ with endpoints $p$ and $p'$; alternatively we can think of this as a framed point in $T$. Consider $x_i$, for $i=1,2$, a properly embedded arc in $M_i$ with endpoints $p$ and $p'$; additionally consider $\gamma$ an embedded arc in $T$ with endpoints $p$ and $p'$, and call $\gamma'$ the arc obtained from $\gamma$ by switching endpoints (see Figure \ref{fig:gamma}). From these, we create two curves $\Gamma(x_1,\gamma,x_2)=x_1\cup \gamma \cup x_2$ and $\Gamma(x_1,\gamma',x_2)=x_1\cup\gamma'\cup x_2$.

\begin{proposition}
For any $x_1,\gamma,x_2$ as above, we have $$(A^4-1)\left(\Gamma(x_1,\gamma,x_2)-\Gamma(x_1,\gamma',x_2)\right)=0.$$
\end{proposition}
\begin{proof}

\begin{figure}
    \centering
      \begin{minipage}{.35\textwidth}
 \centering    \begin{tikzpicture}[y=1cm, x=1cm, yscale=0.25,xscale=0.3, every node/.append style={scale=0.45}, inner sep=0pt, outer sep=0pt]
  \path[draw=black,line cap=butt,line join=miter,line width=0.0265cm] (3.0476, 15.5247) -- (6.7537, 20.0415) -- (17.5727, 20.0455) -- (14.5475, 15.5203) -- cycle;

  \path[draw=blue,line cap=butt,line join=miter,line width=0.0665cm,miter limit=4.0] (5.3774, 17.1049) -- (14.2404, 17.1212);

  \path[draw=blue,line cap=butt,line join=miter,line width=0.0665cm,miter limit=4.0] (8.3789, 24.5171) -- (8.4088, 17.5574);

  \path[draw=blue,line cap=butt,line join=miter,line width=0.0665cm,miter limit=4.0] (11.6818, 24.5182) -- (11.7029, 17.6365);

  \path[draw=blue,line cap=butt,line join=miter,line width=0.0665cm,miter limit=4.0] (8.4172, 16.5481) -- (8.419, 13.0696);

  \path[draw=blue,line cap=butt,line join=miter,line width=0.0665cm,miter limit=4.0] (11.6987, 16.5775) -- (11.7016, 13.2747);

\end{tikzpicture}
\caption{ }\label{fig:firstlink}
   \end{minipage}   
  \begin{minipage}{.35\textwidth}
 \centering    

\begin{tikzpicture}[y=1cm, x=1cm, yscale=0.25,xscale=0.3, every node/.append style={scale=0.45}, inner sep=0pt, outer sep=0pt]
  \path[draw=black,line cap=butt,line join=miter,line width=0.0265cm] (3.0476, 15.5247) -- (6.7537, 20.0415) -- (17.5727, 20.0455) -- (14.5475, 15.5203) -- cycle;

  \path[draw=blue,line cap=butt,line join=miter,line width=0.0565cm,miter limit=4.0] (8.3789, 24.5171) -- (8.4088, 17.5574);

  \path[draw=blue,line cap=butt,line join=miter,line width=0.0665cm,miter limit=4.0] (11.6818, 24.5182) -- (11.7029, 17.6365);

  \path[draw=blue,line cap=butt,line join=miter,line width=0.0665cm,miter limit=4.0] (11.6987, 16.5775) -- (11.7016, 13.2747);

  \path[draw=blue,line cap=butt,line join=miter,line width=0.0665cm,miter limit=4.0] (8.3789, 24.5171) -- (8.419, 13.0696) -- (8.419, 13.0696);

  \path[draw=blue,line cap=butt,line join=miter,line width=0.0765cm,miter limit=4.0] (11.6818, 24.5182) -- (11.7016, 13.2747) -- (11.7016, 13.2747);

  \path[draw=blue,line cap=butt,line join=miter,line width=0.0665cm,miter limit=4.0] (5.4897, 18.1495) -- (7.688, 18.1695);

  \path[draw=blue,line cap=butt,line join=miter,line width=0.0665cm,miter limit=4.0] (8.9991, 18.1677) -- (11.02, 18.1645);

  \path[draw=blue,line cap=butt,line join=miter,line width=0.0665cm,miter limit=4.0] (12.3499, 18.1619) -- (15.0782, 18.1459);

\end{tikzpicture}\caption{ }\label{fig:secondlink}
   \end{minipage}

\end{figure}
Consider the links depicted in Figures \ref{fig:firstlink} and \ref{fig:secondlink}; the plane represents a portion of the separating torus, the curve intersecting it twice represent $x_1\cup x_2$ and the curve lying on it is isotopic to the closure of $\gamma$ (meaning the union of $\gamma$ with the small segment with endpoints $p$ and $p'$ as above). These two links are isotopic; indeed, the complement of the closure of $\gamma$ is an annulus which we can use to push the closure of $\gamma$ "to the other side" of $x_1\cup x_2$. 

We now expand the two links using Kauffman relations. For Figure \ref{fig:firstlink} we have

$$\vcenter{\hbox{\includegraphics[width=0.12\textwidth]{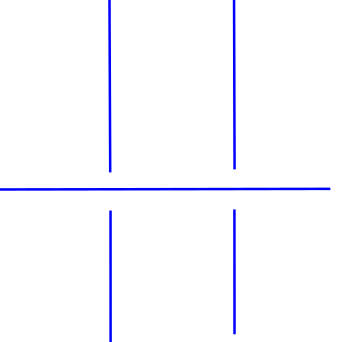}}}= A^2\vcenter{\hbox{\includegraphics[width=0.12\textwidth]{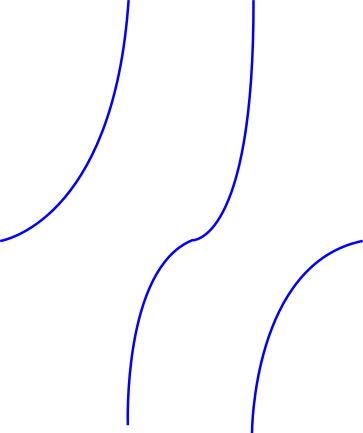}}}+A^{-2}\vcenter{\hbox{\includegraphics[width=0.12\textwidth]{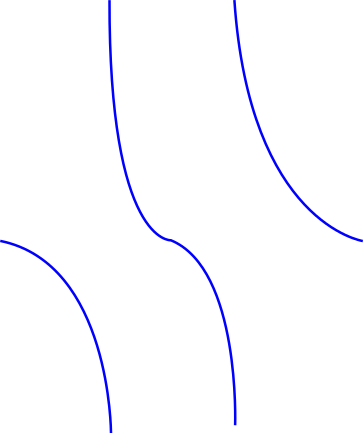}}}+\vcenter{\hbox{\includegraphics[width=0.12\textwidth]{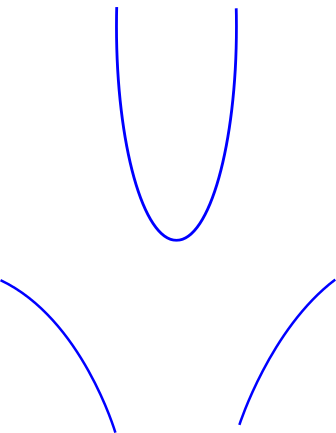}}}+ \vcenter{\hbox{\includegraphics[width=0.12\textwidth]{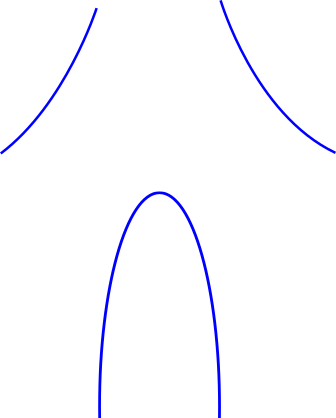}}}$$

and for Figure \ref{fig:secondlink} we have

$$\vcenter{\hbox{\includegraphics[width=0.12\textwidth]{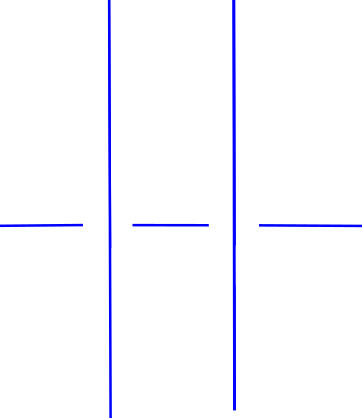}}}= A^2\vcenter{\hbox{\includegraphics[width=0.12\textwidth]{figures/summand1}}}+A^{-2}\vcenter{\hbox{\includegraphics[width=0.12\textwidth]{figures/summand2}}}+\vcenter{\hbox{\includegraphics[width=0.12\textwidth]{figures/summand3}}}+ \vcenter{\hbox{\includegraphics[width=0.12\textwidth]{figures/summand4}}}$$

These two must be equal in the Skein module, which means that 

$$(A^2-A^{-2})\left(\vcenter{\hbox{\includegraphics[width=0.12\textwidth]{figures/summand1}}}- \vcenter{\hbox{\includegraphics[width=0.12\textwidth]{figures/summand2}}} \right) =0$$

the diagrams in this equality correspond exactly to $\Gamma(x_1,\gamma,x_2)$ and $\Gamma(x_1,\gamma',x_2)$.
\end{proof}

We now give some auxiliary results about $M_2(\C)$.

Denote with $\mathcal{D}$ the subalgebra of $M_2(\C)$ given by diagonal matrices, with $\mathcal{U}$ and $\mathcal{L}$ the subalgebras given by upper and lower triangular matrices respectively, and with $\mathcal{J}$ the subalgebra of $M_2(\C)$ given by upper triangular matrices with equal diagonal entries.

\begin{lemma}\label{lem:triangularalgebra}
The only subalgebras of $M_2(\C)$ containing $\mathcal{U}$ are $\mathcal{U}$ and $M_2(\C)$; the only subalgebras of $M_2(\C)$ containing $\mathcal{L}$ are $\mathcal{L}$ and $M_2(\C)$
\end{lemma}
\begin{proof}
As vector spaces, $\mathcal{U}$ and $\mathcal{L}$ have dimension $3,$ while $M_2(\C)$ has dimension $4.$ So a subspace of $M_2(\C)$ containing $\mathcal{U}$ (resp. $\mathcal{L}$) is either $\mathcal{U}$ (resp. $\mathcal{L}$) or $M_2(\C).$
\end{proof}

\begin{lemma}\label{lem:diagonalalgebra}
The only subalgebras of $M_2(\C)$ containing $\mathcal{D}$ are $\mathcal{D}$, $\mathcal{U}$, $\mathcal{L}$ and $M_2(\C)$.
\end{lemma}
\begin{proof}
Let $A$ be a subalgebra of $M_2(\C)$ that contains $\mathcal{D}$ and suppose there is $X\in A\setminus \mathcal{D}$; we can assume that there is one such $X$ of the form $\begin{pmatrix}
0 & a \\
b & 0
\end{pmatrix}$ with either $a\neq 0$ or $b\neq0$. Suppose $a=0$; then clearly $A$ must contain $\mathcal{U}$. Viceversa if $b=0$ clearly $A$ must contain $\mathcal{L}$. In either case, Lemma \ref{lem:triangularalgebra} completes the proof. Therefore, assume that $a\neq 0$ and $b\neq 0$. However, because $X$ and $E_1:=\begin{pmatrix}
1 & 0 \\
0 & 0
\end{pmatrix}$ both belong to $A$, so must $E_1X=\begin{pmatrix}
0 & a \\
0 & 0
\end{pmatrix}$ and $XE_1=\begin{pmatrix}
0 & 0 \\
b & 0
\end{pmatrix}$, which implies that $A$ must be $M_2(\C)$.
\end{proof}

\begin{lemma}\label{lem:nonsemisimplealgebra}
The only subalgebras of $M_2(\C)$ containing $\mathcal{J}$ are $\mathcal{J}$, $\mathcal{U}$ and $M_2(\C)$. 
\end{lemma}

\begin{proof}
Let $A$ be a subalgebra of $M_2(\C)$ containing $\mathcal{J}$, and let $X\in A\setminus \mathcal{J}$. Notice that $U_1:=\begin{pmatrix}
0 & 1 \\
0 & 0
\end{pmatrix}$ belongs to $\mathcal{J}$, as does the identity matrix. If $X=\begin{pmatrix}
a& b \\
c& d
\end{pmatrix}$, then $U_1X=\begin{pmatrix}
c & d \\
0 & 0
\end{pmatrix}$ must belong to $A$ also. If $c\neq 0$, then $E_1$ must belong to $A$, which in turn implies that $A$ must contain $\mathcal{U}$, which would conclude the proof. If instead $c=0$, it means that $X=\begin{pmatrix}
a& b \\
0& d
\end{pmatrix}$ with $a\neq d$ (otherwise $X\in \mathcal{J}$). Therefore $A$ must contain the diagonal matrix $X-bU_1$, and thus must contain $\mathcal{D}$; Lemma \ref{lem:diagonalalgebra} concludes the proof.
\end{proof}

\begin{lemma}\label{lem:repcriterion}
If $\rho$ satisfies the conditions of Theorem \ref{thm:torsiontorus}, then there exist some curves $x_i$ in $M_i$ ($i=1,2$) and $\gamma$ a curve in $T$ such that 
\begin{displaymath}
\mathrm{Tr}(\rho(x_1x_2\gamma))\neq \mathrm{Tr}(\rho(x_1\gamma x_2))
\end{displaymath}
\end{lemma}
\begin{proof}
Suppose by contradiction that 
$\mathrm{Tr}(\rho(x_1x_2\gamma))= \mathrm{Tr}(\rho(x_1\gamma x_2))$ for all $x_1,x_2,\gamma$; in other words that $\mathrm{Tr}(A_1 A_2 A_3)=\mathrm{Tr}(A_1A_3A_2)$ for all $A_1\in \rho(\pi_1(M_1))$, $A_2\in \rho(\pi_1(M_2))$ and $A_3\in\rho(\pi_1(T))$; then by linearity this equality must also hold for all $A_1\in \C[\rho(\pi_1(M_1))]$, $A_2\in \C[\rho(\pi_1(M_2))]$ and $A_3\in\C[\rho(\pi_1(T))]$ where $\C[G]$ is the subalgebra of $M_2(\C)$ generated by the group $G\subseteq SL_2(\C)$. 

Our aim is now to show that we can always find $A_1,A_2,A_3$ such that this equality does not hold.
Let us write $\rho_T$ for the restriction of $\rho$ to $\pi_1(T)\subset \pi_1(M).$
The representation $\rho_T$ is either semisimple or not. 

\emph{Case 1: $\rho_T$ is semisimple.} After possibly substituting $\rho$ with a conjugate representation,  there exists $\gamma\subseteq T$ such that $\rho_T(\gamma)=\begin{pmatrix}
\lambda & 0 \\
0 & \lambda^{-1}
\end{pmatrix}$, and because $\rho_T$ is non-central we can assume that $\lambda\neq \pm 1$. In this case $\C[\rho_T(\pi_1(T))]$ must contain the algebra $\mathcal{D}$ of diagonal matrices. To see this, notice that the identity matrix $I$ is in $\C[\rho_T(\pi_1(T))]$ and because $\lambda\neq \lambda^{-1}$, an appropriate linear combination of $I$ and $\rho_T(\gamma)$ will give both the matrices $\begin{pmatrix}
1 & 0 \\
0 & 0
\end{pmatrix}$ and $\begin{pmatrix}
0 & 0 \\
0 & 1
\end{pmatrix}$ which generate $\mathcal{D}$. Notice that the listed properties of $\rho$ are invariant by conjugation.

Therefore, the algebras $\C[\rho_i(\pi_1(M_i))]$ must contain $\mathcal{D}$; since neither of them can be abelian, they must be strictly bigger, therefore by Lemma \ref{lem:diagonalalgebra} they are either $\mathcal{L},\mathcal{U}$ or $M_2(\C)$. They cannot both be $\mathcal{U}$, since this would imply that $\C[\rho(\pi_1(M))]=\mathcal{U}$ (which would imply that $\rho$ is reducible); likewise they cannot both be $\mathcal{L}$. Suppose without loss of generality that $\mathcal{U}\subseteq\C[\rho_1(\pi_1(M_1))]$ and $\mathcal{L}\subseteq\C[\rho(\pi_1(M_2))]$. Then pick 
$$A_1=\begin{pmatrix}
0 & 1 \\ 0 & 0
\end{pmatrix}, \ A_2=\begin{pmatrix}
0 & 0 \\ 1 & 0
\end{pmatrix} \ \textrm{and} \ A_3=\begin{pmatrix}
1 & 0 \\ 0 & 0
\end{pmatrix}$$
 and a simple calculation shows that 
 $$0=\Tr(A_1A_3A_2)\neq \Tr(A_1A_2A_3)=1.$$

\emph{Case 2: $\rho_T$ is not semisimple} After possibly conjugating, there must be $\gamma\subseteq T$ such that $\rho_T(\gamma)=\pm\begin{pmatrix}
1 & a \\
0 & 1
\end{pmatrix}$ with $a\neq 0$ (again, because we know that $\rho_T$ is not central). In this case a similar argument as Case 1 shows that $\C[\rho_T(\pi_1(T))]$ must contain the algebra $\mathcal{J}$.

The line of reasoning now is the same as Case 1 using Lemma \ref{lem:nonsemisimplealgebra}; we can assume that $\mathcal{U}\subseteq\C[\rho_1(\pi_1(M_1))]$ and $\mathcal{L}\subseteq\C[\rho(\pi_1(M_2))]$. We can now pick 
$$A_1=\begin{pmatrix}
1 & 0 \\ 0 & 0
\end{pmatrix} \ A_2=\begin{pmatrix}
0 & 0 \\ 1 & 0
\end{pmatrix}\ \textrm{and} \ A_3=\begin{pmatrix}
0 & 1 \\ 0 & 0
\end{pmatrix}$$ and show that 
$$1=\Tr(A_1A_3A_2)\neq \Tr(A_1A_2A_3)=0.$$
\end{proof}

\begin{proposition}
	Under the assumptions of Theorem \ref{thm:torsiontorus}, there exist $x_1,\gamma,x_2$ such that $\Gamma(x_1,\gamma,x_2)-\Gamma(x_1,\gamma',x_2)\neq 0$ in $\mathcal{S}(M)$.
\end{proposition}
\begin{proof}
	Consider the map $\Psi: \mathcal{S}(M)\ra \C[X(M)],$ which we recalled in the introduction, sending a link $L_1\sqcup \dots\sqcup L_n$ to $t_{L_1}\ldots t_{L_n}$. Then $$\Psi(\Gamma(x_1,\gamma,x_2))(\rho)=-\Tr(\rho(\Gamma(x_1,\gamma,x_2))=-\Tr(\rho_1(x_1)\rho_T(\gamma)\rho_2(x_2)),$$ where with a slight abuse of notation we use the symbol $x_i$ for the element $\pi_1(M_i)$ obtained by closing the arc $x_i$; likewise for $\gamma$. Similarly, $\Psi(\Gamma(x_1,\gamma',x_2)(\rho)=-\Tr(\rho(\Gamma(x_1,\gamma',x_2))=-\Tr(\rho_1(x_1)\rho_2(x_2)\rho_T(\gamma))$. We now use Lemma \ref{lem:repcriterion} to obtain that there must be $x_1,x_2,\gamma$ such that $\Psi(\Gamma(x_1,\gamma,x_2)-\Gamma(x_1,\gamma',x_2))(\rho)\neq 0$, which implies that $\Gamma(x_1,\gamma,x_2)-\Gamma(x_1,\gamma',x_2)$ is also not zero.
\end{proof}

As an immediate corollary of this criterion, we can prove that splicings of non-trivial knots have torsion in their skein module. First, the definition of splicing.

\begin{definition}
	Let $K_1,K_2\subseteq S^3$ be two knots, and $E_{K_1},E_{K_2}$ their respective exterior (i.e. $S^3$ with an open tubular neighborhood of the knots removed). The \emph{splicing} of $K_1$ and $K_2$ is the manifold $E_{K_1,K_2}$ obtained by gluing $E_{K_1}$ and $E_{K_2}$ along their torus boundary by identifying the longitude of $K_1$ to the meridian of $K_2$ and the meridian of $K_1$ to the longitude of $K_2$. In particular $E_{K_1,K_2}$ is an integral homology sphere.
\end{definition}
\begin{corollary}
	Let $K_1,K_2\subseteq S^3$ be two non-trivial knots. Then $\Sk(E_{K_1,K_2})$ has $(A\pm1)$-torsion.
\end{corollary}
\begin{proof}
	Let $T\subseteq E_{K_1,K_2}$ be the torus arising from the boundaries of $E_{K_1}$ and $E_{K_2}$; clearly it is separating. We wish to apply the criterion of Theorem \ref{thm:torsiontorus}; to do so, we need to find a representation $\rho:\pi_1(E_{K_1,K_2})\ra \slC$. Using \cite[Theorem 8.3]{zent18} provides an irreducible representation $\rho:\pi_1(E_{K_1,K_2})\ra SU(2)$ that restricts to irreducible representations of $\pi_1(E_{K_i})$. The proof, roughly speaking,  finds two representations on $\pi_1(E_{K_1})$ and $\pi_1(E_{K_2})$ whose restrictions to $T$ agree and are not equal to the restriction of an abelian representation of either $\pi_1(E_{K_1})$ or $\pi_1(E_{K_2})$; this implies in particular that $\rho$ restricts to a non-central representation of $\pi_1(T)$. 
	
	Then, because $SU(2)$ embeds in $\slC$, we also obtain a representation $\rho$ with values in $\slC$ with all the same properties. Therefore we can apply Theorem \ref{thm:torsiontorus} to conclude. 
\end{proof}
 
\subsection{The case of non-separating tori}
\label{sec:torsion_nseptori}
Throughout this subsection $M$ is a compact connected manifold with a non-separating incompressible torus $T$. As before, for $\rho:\pi_1(M) \longrightarrow \slC,$ let $\rho_T$ denote the representation $\rho|_{\pi_1(T)}.$ Let $D$ be the following subgroup of $\slC:$ 
$$D=\lbrace \begin{pmatrix}
\lambda & 0 \\ 0 & \lambda^{-1}
\end{pmatrix} \ | \ \lambda \in \C^* \rbrace \cup \lbrace \begin{pmatrix}
0 & -\lambda \\ \lambda^{-1} & 0
\end{pmatrix} \ | \ \lambda \in \C^* \rbrace.$$
We say that a representation $\rho: \pi_1(M)\longrightarrow \slC$ is of \textit{dihedral type with respect to $T$} if either $\rho$ takes value in $D$ and $\rho(\gamma)$ is antidiagonal for any loop whose algebraic intersection with $T$ is odd, or if $\rho$ is conjugate to a representation with that property.
\begin{theorem}\label{thm:nsep-torus}
Suppose there exists a representation $\rho:\pi_1(M)\ra SL(2,\C)$ such that $\rho_T$ is non-central and either of the following two conditions holds:
\begin{itemize}
	\item[(1)]$\rho_T$ is semi-simple and $\rho$ is not of dihedral type with respect to $T.$
	\item[(2)]$\rho_T$ is not semi-simple and $\rho$ is irreducible.
\end{itemize} Then, $\Sk(M,\Z[A^{\pm 1}])$ has $(A\pm 1)$-torsion.
\end{theorem}
\begin{remark}
	\label{rk:nsep-torus-withH1} Let us note that if the image of the inclusion map $\iota : H_1(T,\Z)\longrightarrow H_1(M,\Z)$ is not $2$-torsion, then there is an abelian representation $\rho:\pi_1(M)\longrightarrow \slC$ that satisfies condition (1).
	
	Indeed, since $\iota(H_1(T,\Z))$ is not $2$-torsion, there will be an abelian representation of $\pi_1(M)$ which will be non-central (and semi-simple) on $\pi_1(T),$ and abelian representations are not of dihedral type.
\end{remark}
\begin{proof}
	By \cite[Figure 4.1]{Prz99}, it suffices to find a representation $\rho:\pi_1(M)\longrightarrow \slC$ and loops $\delta\in \pi_1(T),$ $\gamma \in \pi_1(M)$ with $\gamma$ intersecting $T$ geometrically once, such that 
	$$\tr (\rho(\gamma\delta))\neq \tr (\rho(\gamma^{-1} \delta)).$$
	
	\underline{Case 1:} We fix an element $\gamma$ as above. Up to conjugation we assume that $\rho_T$ has diagonal image, and let us write 
	$$\rho(\gamma)=\begin{pmatrix}
	u & v \\ w & z
	\end{pmatrix}.$$
	Note that since $\rho_T$ is non-central, $\C[\rho(\pi_1(T))]$ contains all diagonal matrices.
	If $\tr (\rho(\gamma\delta))= \tr (\rho(\gamma^{-1} \delta))$ for all $\delta \in \pi_1(T),$ then
	$$\tr (\begin{pmatrix}
	u & v \\ w & z
	\end{pmatrix}\begin{pmatrix}
	x & 0 \\ 0 & y
	\end{pmatrix} ) =\tr (\begin{pmatrix}
	z & -v \\ -w & u
	\end{pmatrix}\begin{pmatrix}
	x & 0 \\ 0 & y
	\end{pmatrix} ),$$
	for any $x,y\in \C,$ which gives $u=z.$ However, replacing $\gamma$ by $\gamma\delta',$ for some $\delta'\in \pi_1(T)$ such that $\rho(\delta')\neq \pm I_2,$ we can get that $u=z=0,$ i.e. we can get that $\rho(\gamma)$ is antidiagonal. 
	Therefore, unless $\rho$ is dihedral type with respect to $T,$ there is a choice of $\gamma,\delta$ such that $\tr \rho(\gamma\delta)\neq \tr(\rho(\gamma^{-1}\delta)),$ and thus there is $(A\pm 1)$-torsion in $\mathcal{S}(M).$
	
	\underline{Case 2:} Up to conjugation, we assume that $\rho_T$ takes value in upper triangular matrices, and since $\rho_T$ is non-central, $\C[\rho(\pi_1(T))]=\mathcal{J}.$ Let $\gamma$ be a loop that intersects $T$ geometrically once and write $\rho(\gamma)=\begin{pmatrix}
	u & v \\ w & z
	\end{pmatrix}.$
	If $\tr (\rho(\gamma\delta))\neq \tr (\rho(\gamma^{-1} \delta))$ for all $\delta \in \pi_1(T),$ then
	$$\tr (\begin{pmatrix}
	u & v \\ w & z
	\end{pmatrix}\begin{pmatrix}
	x & y \\ 0 & x
	\end{pmatrix} ) =\tr (\begin{pmatrix}
	z & -v \\ -w & u
	\end{pmatrix}\begin{pmatrix}
	x & y \\ 0 & x
	\end{pmatrix} ),$$
	for all $x,y\in \C,$ which gives $w=0.$ However, if $w=0$ for all $\gamma$ intersecting $T$ geometrically once, since $\pi_1(M)$ is generated by $\pi_1(T)$ and such elements, we would have $\rho(\pi_1(M))\subset \mathcal{U},$ which would contradict the irreducibility of $\rho.$
\end{proof}
\section{The strong finiteness conjecture for some Seifert manifolds with non-orientable base}
\label{sec:finiteness}
In this section, we will prove Conjecture \ref{conj:finiteness_boundary} for some Seifert manifolds; namely those with base a M\"obius band with one exceptional fiber and the one with base a M\"obius band with a disk removed and no exceptional fibers. The aim will be to apply Theorem \ref{thm:largeX(M)_nclosed} to these manifolds to show that their skein module contains torsion.

The proofs will rely on manipulating arrowed diagrams, which are a convenient way of representing framed links in $F\maybesimtimes S^1,$ (which means either $F\times S^1$ if $F$ is orientable, or $F\simtimes S^1$ if $F$ is not); this was introduced in \cite{MD09} for orientable surfaces and in \cite{Mro11} for non-orientable ones.

 An arrowed diagram $D$ on $F$ is a $4$-valent graph on $F$ with under-/overcrossing information at vertices, and with arrows possibly added on the strands. It gives rise to a framed link in $F\times S^1$ in the following way: if there are no arrows, we can just interpret the diagram as the diagram of a framed link in $F\times S^1,$ otherwise, if there are arrows we modify the link by making it wrap around the $S^1$ factor positively in the direction of the arrow.

Arrowed diagrams connected by the usual Reidemeister moves (on parts of the diagram without arrows) represent the same links, but in addition, we have the moves:
\begin{center}
	\def\svgwidth{0.5 \columnwidth}
\begingroup%
  \makeatletter%
  \providecommand\color[2][]{%
    \errmessage{(Inkscape) Color is used for the text in Inkscape, but the package 'color.sty' is not loaded}%
    \renewcommand\color[2][]{}%
  }%
  \providecommand\transparent[1]{%
    \errmessage{(Inkscape) Transparency is used (non-zero) for the text in Inkscape, but the package 'transparent.sty' is not loaded}%
    \renewcommand\transparent[1]{}%
  }%
  \providecommand\rotatebox[2]{#2}%
  \ifx\svgwidth\undefined%
    \setlength{\unitlength}{221.64836426bp}%
    \ifx\svgscale\undefined%
      \relax%
    \else%
      \setlength{\unitlength}{\unitlength * \real{\svgscale}}%
    \fi%
  \else%
    \setlength{\unitlength}{\svgwidth}%
  \fi%
  \global\let\svgwidth\undefined%
  \global\let\svgscale\undefined%
  \makeatother%
  \begin{picture}(1,0.1979595)%
    \put(0,0){\includegraphics[width=\unitlength]{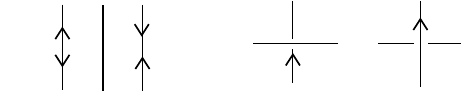}}%
    \put(0.1511898,0.08198808){\color[rgb]{0,0,0}\makebox(0,0)[lb]{\smash{$\sim$}}}%
    \put(0.23048952,0.08198808){\color[rgb]{0,0,0}\makebox(0,0)[lb]{\smash{$\sim$}}}%
    \put(-0.00102922,0.0792536){\color[rgb]{0,0,0}\makebox(0,0)[lb]{\smash{$(R_4)$}}}%
    \put(0.43557502,0.0792536){\color[rgb]{0,0,0}\makebox(0,0)[lb]{\smash{$(R_5)$}}}%
    \put(0.74001303,0.07469615){\color[rgb]{0,0,0}\makebox(0,0)[lb]{\smash{$\sim$}}}%
  \end{picture}%
\endgroup%

\end{center}

Moreover, for conveniency we will enrich arrowed diagrams to include dots, which represents a component of a framed link which is parallel to a fiber $S^1.$ Dots can be expressed in terms of arrowed unknots in the following way:
\begin{center}
	\def\svgwidth{0.4 \columnwidth}
\begingroup%
  \makeatletter%
  \providecommand\color[2][]{%
    \errmessage{(Inkscape) Color is used for the text in Inkscape, but the package 'color.sty' is not loaded}%
    \renewcommand\color[2][]{}%
  }%
  \providecommand\transparent[1]{%
    \errmessage{(Inkscape) Transparency is used (non-zero) for the text in Inkscape, but the package 'transparent.sty' is not loaded}%
    \renewcommand\transparent[1]{}%
  }%
  \providecommand\rotatebox[2]{#2}%
  \newcommand*\fsize{\dimexpr\f@size pt\relax}%
  \newcommand*\lineheight[1]{\fontsize{\fsize}{#1\fsize}\selectfont}%
  \ifx\svgwidth\undefined%
    \setlength{\unitlength}{110.95725352bp}%
    \ifx\svgscale\undefined%
      \relax%
    \else%
      \setlength{\unitlength}{\unitlength * \real{\svgscale}}%
    \fi%
  \else%
    \setlength{\unitlength}{\svgwidth}%
  \fi%
  \global\let\svgwidth\undefined%
  \global\let\svgscale\undefined%
  \makeatother%
  \begin{picture}(1,0.26880558)%
    \lineheight{1}%
    \setlength\tabcolsep{0pt}%
    \put(0,0){\includegraphics[width=\unitlength,page=1]{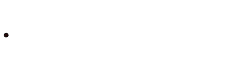}}%
    \put(0.0835308,0.10736947){\color[rgb]{0.16862745,0,0}\makebox(0,0)[lt]{\lineheight{1.25}\smash{\begin{tabular}[t]{l}$=-A^3$\end{tabular}}}}%
    \put(0,0){\includegraphics[width=\unitlength,page=2]{FiberNotation.pdf}}%
    \put(0.55523564,0.1058672){\color[rgb]{0.16862745,0,0}\makebox(0,0)[lt]{\lineheight{1.25}\smash{\begin{tabular}[t]{l}$=-A^{-3}$\end{tabular}}}}%
    \put(0,0){\includegraphics[width=\unitlength,page=3]{FiberNotation.pdf}}%
  \end{picture}%
\endgroup%

\end{center}

\subsection{The case of \texorpdfstring{$F_{1,2}\simtimes S^1$}{FxS¹}}
\label{sec:Mobiushole}
Consider now $F_{1,2}$ the M\"obius band with one hole (see Figure \ref{fig:F12}). The aim of this section is to prove the strong finiteness conjecture for $\F$; more precisely:

\begin{figure}
	\centering
	\includegraphics[scale=0.25]{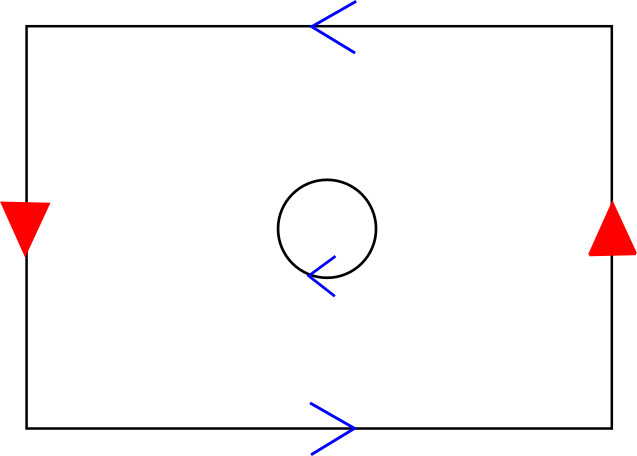}
	\caption{The $1$-holed M\"obius band; the solid red arrows represent a gluing, while the blue arrows represent a choice of orientation on the boundary.}\label{fig:F12}
\end{figure}

\begin{proposition}\label{prop:f1,2case}
	Suppose $c_2$ is a curve in $\partial \F$ that is neither horizontal (i.e. isotopic to a curve in $F_{1,2}\times\{x\}$) nor vertical (i.e. isotopic to $\{x\}\times S^1$). Then there exists $c_1$ a non-horizontal, non-vertical curve in the other component of $\partial \F$ such that $\Sk(\F,\Q(A))$ is finitely generated as a $\Q(A)[c_1,c_2]$-module.
\end{proposition}

Before proving this proposition, it will be convenient to recall the Frohman-Gelca basis \cite{FG00} of the skein algebra of the torus. For $T=S^1\times S^1,$ let $(0,0)_T \in \Sk(T)$ be the empty multicurve. For $p,q\in \Z^2$ coprime, we denote by  $(p,q)_T \in \Sk(T)$ the simple closed curve on $T$ of slope $p/q.$ Finally for $p,q\in \Z$ not both zero, let $(p,q)_T=T_d((p/d,q/d)_T)\in \Sk(T),$ where $d=\gcd(p,q)$ and $T_d(X)$ is the $d$-th Chebychev polynomial of the first type, that is, the unique polynomial in $\Z[X]$ satisfying $T_d(x+x^{-1})=x^d+x^{-d}.$

It is proved in \cite{FG00} that the set of $(p,q)_T$ for $(p,q)\in \Z^2/_{\lbrace \pm 1 \rbrace}$ is a basis of $\Sk(T),$ and that we have the product to sum formula:
$$(p,q)_T\cdot (r,s)_T=A^{ps-qr}(p+r,q+s)_T + A^{qr-ps}(p-r,q-s)_T.$$

\begin{proof}[Proof of Proposition \ref{prop:f1,2case}]
		\begin{figure}
			\begin{minipage}{0.4\textwidth}\centering
			\includegraphics[scale=0.25]{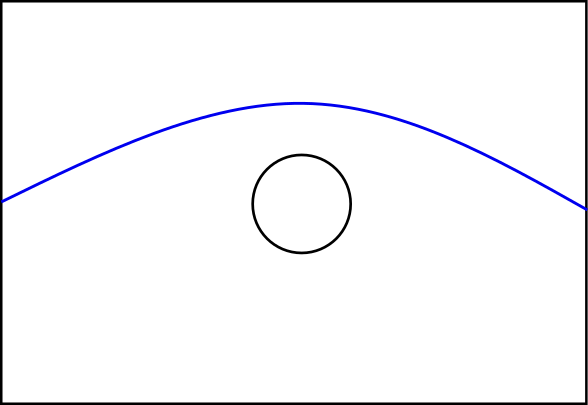}
		\end{minipage}
	\begin{minipage}{0.4\textwidth}\centering
	\includegraphics[scale=0.25]{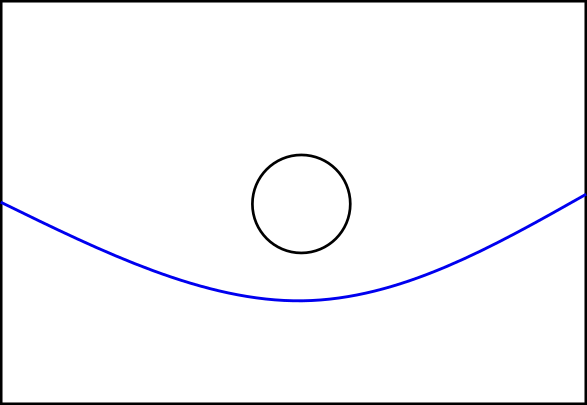}
\end{minipage}
\label{fig:F12x12}
\caption{The curves $x_1$ (on the left) and $x_2$ (on the right)}
		\end{figure}
	First notice that $\Sk(\F,\Q(A))$ is generated, over $\Sk(\partial \F,\Q(A))$, by three elements: the empty skein and the two curves $x_1$ and $x_2$ depicted in Figure \ref{fig:F12x12}. To see this, first notice the skein module of $S\times S^1$, where $S$ is any surface, is generated by arrowed multicurves in $S$. All curves in $F_{1,2}$ either bound a disk, are isotopic to the boundary, or are isotopic to either $x_1$ or $x_2$; notice further all curves isotopic to $x_1$ or $x_2$ intersect each other, therefore a multicurve in $F_{1,2}$ can contain at most one. Therefore $\Sk(\F,\Q(A))$ is generated, as a $\Sk(\partial \F,\Q(A))$-module, by the empty skein and by $x_1,x_2$, possibly with arrows. To see that arrows are not needed we apply a standard trick. Consider the identities
	
		$$\hspace*{1mm}\vcenter{\hbox{\includegraphics[width=0.18\textwidth]{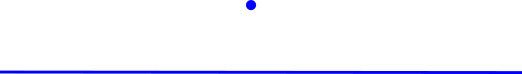}}}=A\hspace*{1mm}\vcenter{\hbox{\includegraphics[width=0.18\textwidth]{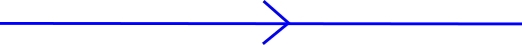}}}+A^{-1}\hspace*{1mm}\vcenter{\hbox{\includegraphics[width=0.18\textwidth]{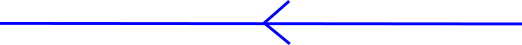}}}$$
	
and 
	
	$$\hspace*{1mm}\vcenter{\hbox{\rotatebox[origin=c]{180}{\includegraphics[width=0.18\textwidth]{figures/linefiber}}}}=A\hspace*{1mm}\vcenter{\hbox{\includegraphics[width=0.18\textwidth]{figures/larrow}}}+A^{-1}\hspace*{1mm}\vcenter{\hbox{\includegraphics[width=0.18\textwidth]{figures/rarrow}}}$$
	
	where the horizontal lines are either $x_1$ or $x_2$. Notice that the left hand sides are actually isotopic (since $x_1$ and $x_2$ are non-separating); this implies that we can invert an arrow at no cost. Then applying this to one of the previous identities we find that

	$$\frac{1}{A+A^{-1}}\hspace*{1mm}\vcenter{\hbox{\includegraphics[width=0.18\textwidth]{figures/linefiber}}}=\hspace*{1mm}\vcenter{\hbox{\includegraphics[width=0.18\textwidth]{figures/rarrow}}}$$
	
	which shows that arrowed curves are not needed to generate over $\Sk(\partial \F,\Q(A))$ (since the fibers can be isotoped into the boundary).
	
	Consider $\mathcal{F}_0= \Sk(\partial \F,\Q(A))\cdot \varnothing$, $\mathcal{F}_1=\Sk(\partial \F,\Q(A))\cdot \{x_1,x_2\}$. Then $\mathcal{F}_0$ is obviously generated over $\Q(A)$ by the elements $((p_1,q_1)_T,(p_2,q_2)_T)\cdot\varnothing$, $p_i,q_i\in \Z$, where $(p_i,q_i)_T$ is a skein element in $\partial_i\F$, here $p$ represents the horizontal coordinate and $q$ the vertical one. An analogous result holds for $\mathcal{F}_1$.
	
	To simplify notation, from now on we write $(a,b,c,d)$ to mean both the element of $\Z^4$, and the skein element $((a,b)_T,(c,d)_T)$; which one is being considered will always be clear from context.
	
The proof of the following lemma is postponed until the end of the section.
	\begin{lemma}\label{lem:F12relations}
		The following relations hold in $\Sk(\F,\Q(A))$:
		\begin{enumerate}
			\item $(0,0,0,1)\cdot\varnothing= (0,1,0,0)\cdot\varnothing$;\label{rel:f12rel1}
			\item $\left(A(1,1,0,0)+A^{-1}(0,0,1,-1)-A(0,0,1,1)-A^{-1}(1,-1,0,0)\right)\cdot\varnothing=0$;\label{rel:f12rel2}
			\item $(0,1,0,0)\cdot x_1=(0,0,0,1)\cdot x_1$;\label{rel:f12rel3}
			\item $(0,1,0,0)\cdot x_2=(0,0,0,1)\cdot x_2$;\label{rel:f12rel4}
			\item $A^2(1,1,0,0)\cdot x_1-A^{-2}(1,-1,0,0)\cdot x_1=(A-A^{-1})(0,0,0,1)\cdot x_2$;\label{rel:f12rel5}
			\item $A^2(1,1,0,0)\cdot x_2-A^{-2}(1,-1,0,0)\cdot x_2=(A-A^{-1})(0,0,0,1)\cdot x_1$;\label{rel:f12rel6}
		\end{enumerate}
	\end{lemma}

Suppose that that the curve $c_2$ on $\partial_2 \F$ has slope $b_2/a_2,$ that is,  $c_2=(a_2,b_2)_T$ in $\Sk(\partial_2 \F,\Q(A))$; choose $c_1=(a_1,b_1)_T$ such that $a_1>0,b_1<0$ and $1-\frac{b_1}{a_1}$ is greater than $\max\left(\lvert 1+\frac{b_2}{a_2}\rvert, \lvert 1-\frac{b_2}{a_2}\rvert\right)$.
	
	We now work on each $\mathcal{F}_i$ separately and prove that each is finitely generated over $\Q(A)[c_1,c_2]$. We begin with $\mathcal{F}_0$. 
	
	First we prove the following lemma.
	
	\begin{lemma}\label{lem:f0affine}
		The space $\mathcal{F}_0$ is generated, over $\Q(A)$, by elements of the form $((p_1,q_1)_T,(p_2,q_2)_T)$ with $(p_1,q_1,p_2,q_2)$ belonging to a finite union of affine subspaces in $\Z^4$ directed by the subspace generated by $(a_1,b_1,0,0)$ and $(0,0,a_2,b_2)$.
	\end{lemma} 
	\begin{proof}
		Define the following auxiliary complexities:
		
		\begin{gather*}
			c_1(u,v,w,z)=\lvert a_1v-b_1u\rvert;\\
			c_2(u,v,w,z)=\lvert a_2z-b_2w\rvert;\\
			c(u,v,w,z)=\frac{c_1(u,v,w,z)}{a_1}+\frac{c_2(u,v,w,z)}{a_2};
		\end{gather*}
		then the complexity function $\mathcal{C}$ is defined as $$\mathcal{C}(u,v,w,z)=(c(u,v,w,z),-c_1(u,v,w,z))$$ with lexicographic ordering. With a slight abuse of notation, if $x\in \Sk(\partial \F,\Q(A))$ is equal to $(p_1,q_1,p_2,q_2)$ we write $\mathcal{C}(x)$ to denote $\mathcal{C}(p_1,q_1,p_2,q_2)$ (and similarly for $c$, $c_1$ and $c_2$).
		
		Notice that because of the way $\mathcal{C}$ is defined, given $\mathcal{C}(x)$ there is only a finite number of values of $\mathcal{C}(x')$ such that $\mathcal{C}(x')<\mathcal{C}(x)$: this is because $c_2(x)\leq a_2c(x)$. Furthermore, notice that if $x=(u,v,w,z)$ is such that $c_1(x)=d$, then for any $k\in \Z$, if $x'=(u-kb_1,v-ka_1,w,z)$ is also such that $c_1(x)=d$; viceversa if $x$ and $x'$ are such that $c_1(x)=c_1(x')$, their first two components must satisfy the above relation for some $k$. This implies that for a given pair $(d_1,d_2)$, the space of $x\in \Z^4$ such that $\mathcal{C}(x)=(d_1,-d_2)$ is an affine space directed by the subspace generated by $(a_1,b_1,0,0)$ and $(0,0,a_2,b_2)$.
		
		Now the fact that $\mathcal{F}_0$ is generated by  elements of the form $(p_1,q_1,p_2,q_2)$ with $(p_1,p_2,q_1,q_2)$ belonging to a finite union of affine subspaces in $\Z^4$ directed by $(a_1,b_1,0,0)$ and $(0,0,a_2,b_2)$ is contained in the following two lemmas, whose proof is postponed to the end of the section.
		
		\begin{lemma}\label{lem:reduceC2}
			Suppose $x$ is such that $c_2(x)>2a_2$. Then $x\cdot \varnothing$ is a linear combination of elements of the form $x'\cdot \varnothing$, with $\mathcal{C}(x')<\mathcal{C}(x)$.
		\end{lemma} 
		\begin{lemma}\label{lem:reduceC1}
			Assume $x$ is such that $c_1(x)>2(a_1-b_1)$, that $a_1>0$, $b_1<0$ and that $1-\frac{b_1}{a_1}> \max\left(\lvert 1+\frac{b_2}{a_2}\rvert,\lvert1-\frac{b_2}{a_2}\rvert\right)$. Then, $x\cdot \varnothing$ is a linear combination of elements of the form $x'\cdot \varnothing$ with $\mathcal{C}(x')<\mathcal{C}(x)$.
		\end{lemma}
		Given 
		Using Lemmas \ref{lem:reduceC2} and \ref{lem:reduceC1} we can see that $\mathcal{F}_0$ is generated by elements of the form $x\cdot \varnothing$ with $\mathcal{C}(x)$ bounded by some constant; therefore the statement of the lemma is satisfied.
	\end{proof}
	
	Now we can prove that for each affine subspace $V$ as above, there is a finite number of elements in $V$ such that the subspace of $\Sk(\F,\Q(A))$ they generate (over $\Q(A)[c_1,c_2]$) is the same as the subspace generated by $V$. Suppose $V$ is given by $\mathcal{C}(x)=(d_1,d_2)$; then as we have seen, a generic element of $V$ is going to be equal to $(u-k_1a_1,v-k_1b_1,w-k_2a_2,z-k_2b_2)$, for some $(u,v,w,z)$ with complexity $(d_1,d_2)$ and for any $(k_1,k_2)$ in $\Z^2$. Then the elements $(u,v,w,z)$, $(u+a_1,v+b_1,w,z)$, $(u,v,w+a_2,z+b_2)$ and $(u+a_1,v+b_1,w+a_2,z+b_2)$ generate, over $\Q(A)[c_1,c_2]$ the same subspace of $\mathcal{F}_0$ as the whole of $V$. To see this, notice that $c_1\cdot (u+a_1,v+b_1,w,z)= A^*(u+2a_1,v+2b_1,w,z)+A^*(u,v,w,z)$ (where $A^*$ is a suitable power of $A$); using this identity (and the analogous one for $c_2$) we can see that the space that those two elements generate contains $(u-k_1a_1,v-k_1b_1,w-k_2a_2,z-k_2b_2)$ for any $(k_1,k_2)\in\Z^2$, so it contains the space generated by $V$. This is enough to prove that $\mathcal{F}_0$ is finitely generated over $\Q(A)[c_1,c_2]$.
	
	We now carry out the same procedure for $\mathcal{F}_1$. The analog of Lemma \ref{lem:f1affine} is the following:
	\begin{lemma}\label{lem:f1affine}
		The space $\mathcal{F}_1$ is generated, over $\Q(A)$, by elements of the form $(p_1,q_1,p_2,q_2)\cdot x_i$ with $i=1,2$ and  $(p_1,p_2,q_1,q_2)$ belonging to a finite union of affine subspaces in $\Z^4$ directed by the subspace generated by $(a_1,b_1,0,0)$ and $(0,0,a_2,b_2)$, where $(a_i,b_i)_T=c_i$.
	\end{lemma} 

The proof proceeds exactly the same as for Lemma \ref{lem:f0affine}; the analogs of Lemmas \ref{lem:reduceC2} and \ref{lem:reduceC1} uses the same complexities but uses Relations \ref{rel:f12rel3}-\ref{rel:f12rel6} instead of Relations \ref{rel:f12rel1} and \ref{rel:f12rel2} in its proof. It is then possible to prove that $\mathcal{F}_1$ is finitely generated over $\Q(A)[c_1,c_2]$ from Lemma \ref{lem:f1affine} using the same reasoning as before.
\end{proof}

Finally we prove the Lemmas used in Proposition \ref{prop:f1,2case}.

\begin{proof}[Proof of Lemma \ref{lem:F12relations}]
	Relation \ref{rel:f12rel1} can be achieved via an isotopy: $(0,1,0,0)\varnothing$ is a fiber in one boundary component and $(0,0,0,1)\varnothing$ is a fiber in the other, therefore they are isotopic. Relations $\ref{rel:f12rel3}$ and $\ref{rel:f12rel4}$ are obtained in the same way (noticing that both $x_1$ and $x_2$ are non-separating in $F_{1,2}$).
	
	Relation \ref{rel:f12rel2} comes from the following identity, a consequence of the arrowed Reidemeister move $R_5$:
	
	$$\vcenter{\hbox{\includegraphics[width=0.2\textwidth]{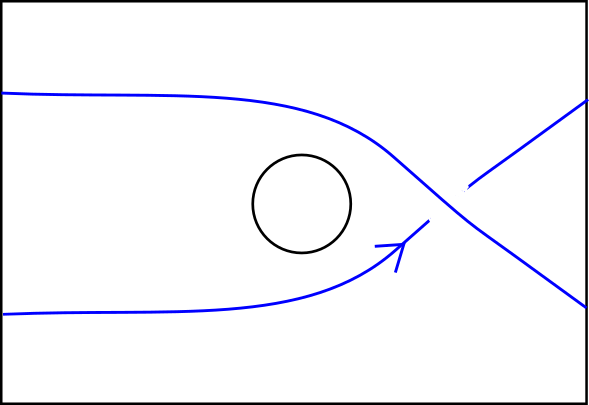}}}=\vcenter{\hbox{\includegraphics[width=0.2\textwidth]{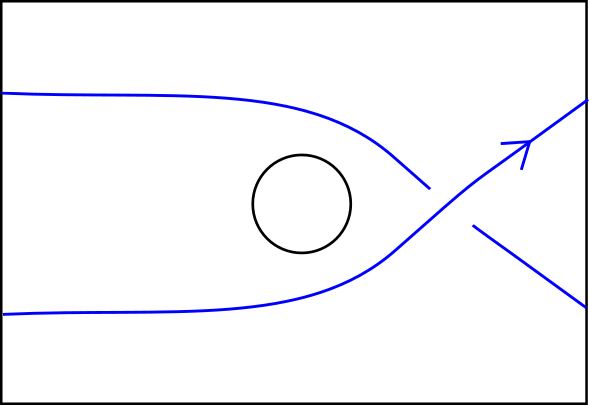}}}$$
	
	Resolving the crossing on both sides using the Kauffman relation gives:
	
		$$A\hspace*{1mm}\vcenter{\hbox{\includegraphics[width=0.18\textwidth]{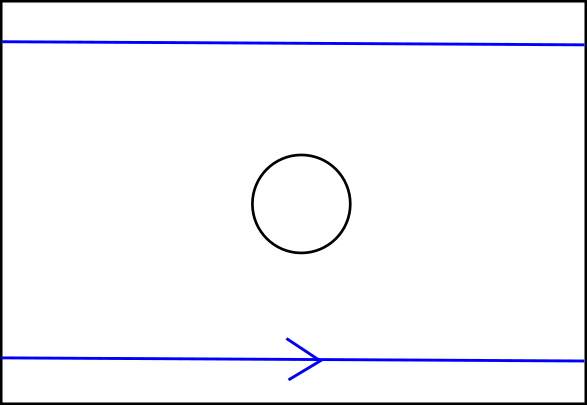}}}+A^{-1}\hspace*{1mm}\vcenter{\hbox{\includegraphics[width=0.18\textwidth]{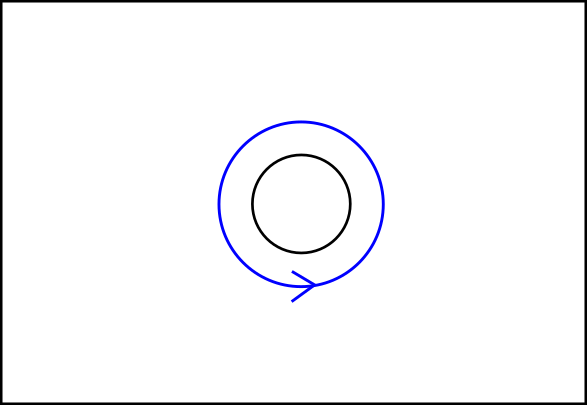}}}=A\hspace*{1mm}\vcenter{\hbox{\includegraphics[width=0.18\textwidth]{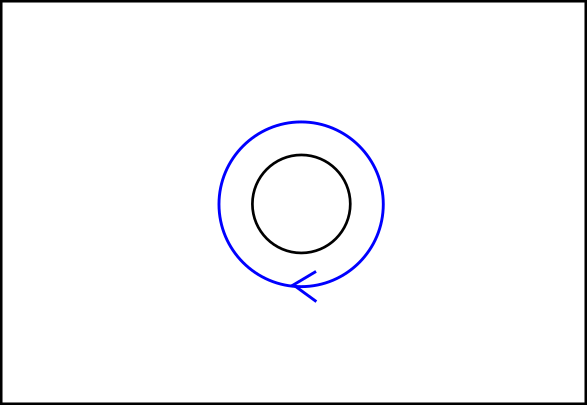}}}+A^{-1}\hspace*{1mm}\vcenter{\hbox{\includegraphics[width=0.18\textwidth]{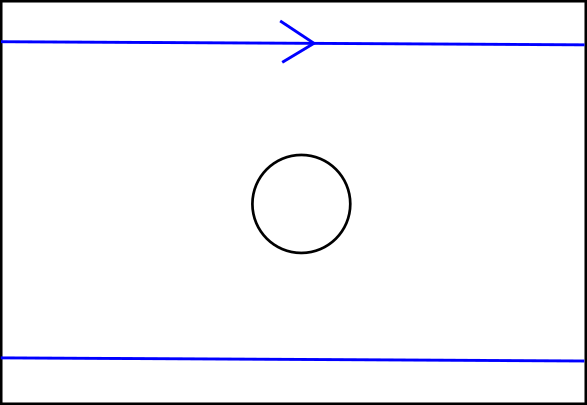}}}$$
	All these curves can be isotoped into the boundary; when written in the Frohman-Gelca basis, the relation reads $$\left(A(1,1,0,0)+A^{-1}(0,0,1,-1)-A(0,0,1,1)-A^{-1}(1,-1,0,0)\right)\cdot\varnothing=0.$$
	
	Relation \ref{rel:f12rel5} follows from the identity 
	
	$$\vcenter{\hbox{\includegraphics[width=0.2\textwidth]{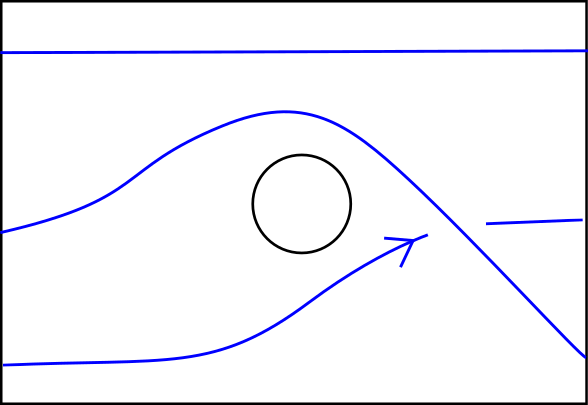}}}=\vcenter{\hbox{\includegraphics[width=0.2\textwidth]{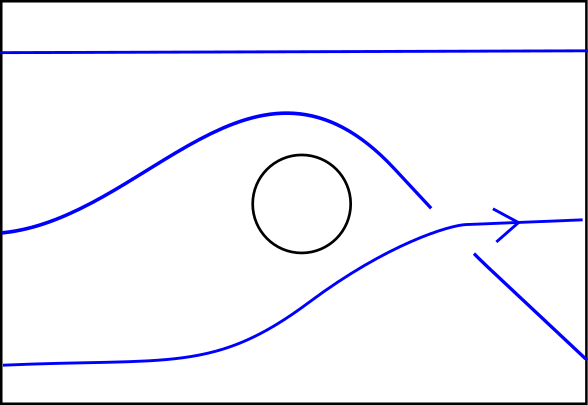}}}$$
	which is once again a consequence of the arrowed Reidemeister move $R_5$. Expanding gives 		$$A\hspace*{1mm}\vcenter{\hbox{\includegraphics[width=0.18\textwidth]{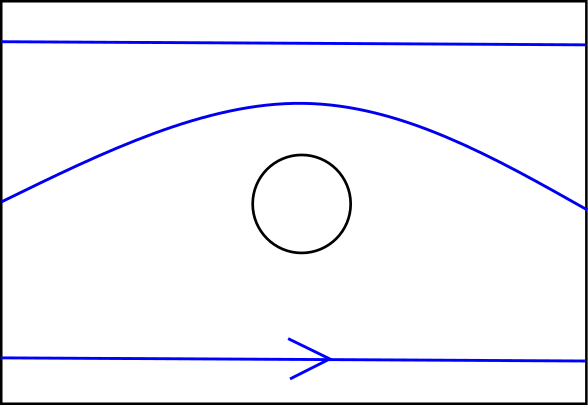}}}+A^{-1}\hspace*{1mm}\vcenter{\hbox{\includegraphics[width=0.18\textwidth]{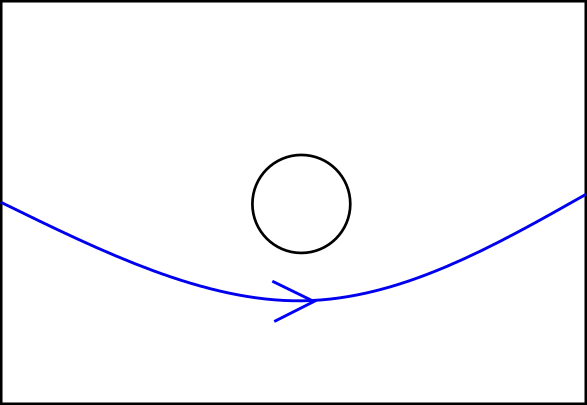}}}=A\hspace*{1mm}\vcenter{\hbox{\includegraphics[width=0.18\textwidth]{figures/f12r4r2}}}+A^{-1}\hspace*{1mm}\vcenter{\hbox{\includegraphics[width=0.18\textwidth]{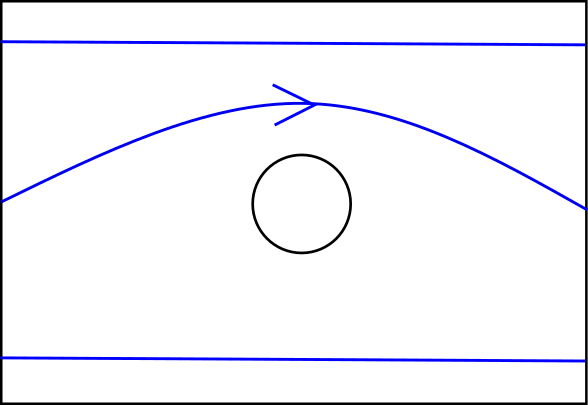}}}$$
	
	This identity can be rewritten as
		$$A\hspace*{1mm}\vcenter{\hbox{\includegraphics[width=0.18\textwidth]{figures/f12r4r1}}}+\frac{A^{-1}-A}{A+A^{-1}}\hspace*{1mm}\vcenter{\hbox{\includegraphics[width=0.18\textwidth]{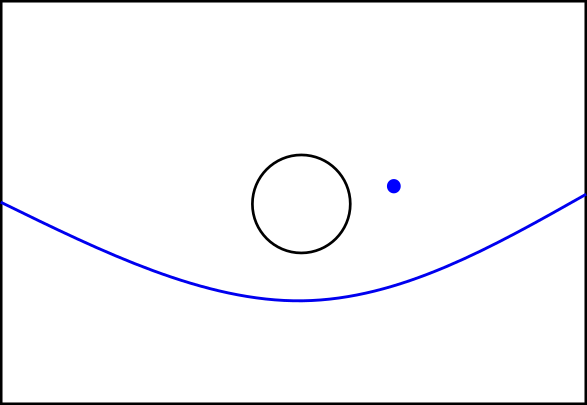}}}=\frac{A^{-1}}{A+A^{-1}}\hspace*{1mm}\vcenter{\hbox{\includegraphics[width=0.18\textwidth]{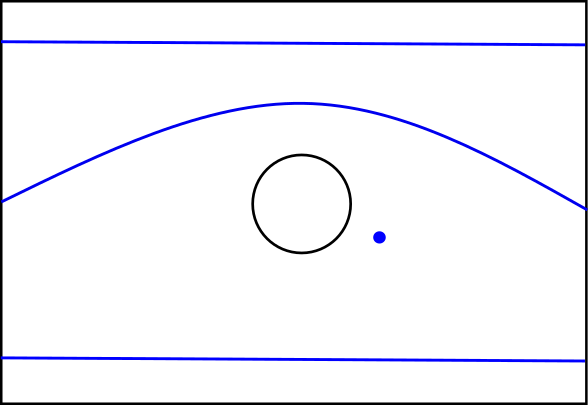}}}$$
		
		using the trick from the start of the proof of Proposition \ref{prop:f1,2case}. 
		
		After rearranging and applying the Kauffman bracket relation to the fiber in the right hand side, we obtain
		
		$$A^2\hspace*{1mm}\vcenter{\hbox{\includegraphics[width=0.18\textwidth]{figures/f12r4r1}}}-A^{-2}\hspace*{1mm}\vcenter{\hbox{\includegraphics[width=0.18\textwidth]{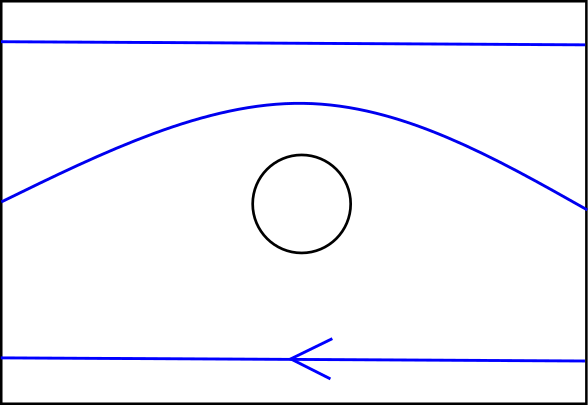}}}=(A-A^{-1})\hspace*{1mm}\vcenter{\hbox{\includegraphics[width=0.18\textwidth]{figures/f12r4r6}}}$$
		
		which gives exactly	$A^2(1,1,0,0)\cdot x_1-A^{-2}(1,-1,0,0)\cdot x_1=(A-A^{-1})(0,0,0,1)\cdot x_2$.
		
		Finally Relation \ref{rel:f12rel6} is the mirror of Relation \ref{rel:f12rel5}.
	
\end{proof}

\begin{proof}[Proof of Lemma \ref{lem:reduceC2}]
	Denote with $x=(u,v,w,z)$ and assume that $a_2z-b_2w\geq 0$ and  that $a_1v-b_1u\geq 0$ (this can be done because $(w,z)_T=(-w,-z)_T$). Then multiplying Relation \ref{rel:f12rel1} of Lemma \ref{lem:F12relations} by the boundary skein element $(u,v,w,z-1)$ allows us to write $x$ as a linear combination of the skein elements $(u,v+1,w,z-1)\cdot \varnothing$, $(u,v-1,w,z-1)\cdot \varnothing$ and $(u,v,w,z-2)\cdot \varnothing$. We show that each of these tuples has lower complexity than $x$. Indeed, $c_2(u,v,w,z-2)=a_2z-b_2w-2a_2<c_2(x)$, and its $c_1$ is unchanged. Furthermore, when compared to $c(u,v,w,z)$, we have that in
	$$c(u,v+1,w,z-1)=\frac{\lvert a_1v-b_1u+a_1\rvert}{a_1}+\frac{a_2z-b_2w-a_2}{a_2}$$
	the first summand can increase by at most $1$, whereas the second summand surely decreases by $1$; this shows that $c(u,v+1,w,z-1)\leq c(u,v,w,z)$ and when there is equality $c_1(u,v+1,w,z-1)> c_1(u,v,w,z)$, which means that $\mathcal{C}(u,v+1,w,z-1)<\mathcal{C}(u,v,w,z)$ in all cases. The remaining case follows from the same line of reasoning.
\end{proof}

\begin{proof}[Proof of Lemma \ref{lem:reduceC1}]
		As in the proof of the previous lemma, we want to fix the signs of $a_2z-b_2w$ and of $a_1v-b_2u$ so that they are positive.
		Take now Relation \ref{rel:f12rel2} of Lemma \ref{lem:F12relations} and multiply it to the left by $(u-1,v-1,w,z)$. After applying the product to sum formula we obtain that $(u,v,w,z)\varnothing$ is a linear combination of the terms
$(u-2,v-2,w,z)\varnothing$,$(u-1,v-1,w+1,z-1)\varnothing$, $(u-1,v-1,w-1,z+1)\varnothing$, $(u-1,v-1,w+1,z+1)\varnothing$, $(u-1,v-1,w-1,z-1)\varnothing$, $(u,v-2,w,z)\varnothing$ and $(u-2,v,w,z)\varnothing$.

	We claim that all these terms have a lower complexity than $(u,v,w,z)$. Indeed, the terms $(u-2,v-2,w,z)$, $(u-2,v,w,z)$ and $(u,v-2,w,z)$ have a lower $c_1$ complexity and same $c_2$ complexity. For example, $c_1(u,v-2,w,z)=\lvert a_1v-b_1u-2a_1\rvert<\lvert a_1v-b_1u\rvert=c_1(u,v,w,z)$ because $c_1(u,v,w,z)>2(a_1-b_1)>2a_1$ (recall that $b_1$ is negative). The other two cases follow the same reasoning. 
	
	The remaining terms are more complicated but they all follow the same reasoning. For example, consider the term $(u-1,v-1,w+1,z+1)$. Its $c_2$ complexity might actually increase by at most $\lvert b_2-a_2\rvert$: $c_2(u-1,v-1,w+1,z+1)=\lvert a_2z-b_2w-a_2+b_2\rvert \leq \lvert a_2z-b_2w\rvert + \lvert b_2-a_2\rvert$. However, a similar calculation shows that its $c_1$ complexity must decrease by at least $a_1-b_1$. Therefore when calculating $\mathcal{C}$, the $c_1$ term decreases by at least $1-\frac{b_1}{a_1}$, while the $c_2$ term increases by at most $\lvert 1-\frac{a_2}{b_2}\rvert$; the assumption on $a_1$ and $b_1$ then implies that $\mathcal{C}$ must decrease. All other summands follow the same reasoning.
	\end{proof}

\subsection{Strong finiteness conjecture for manifolds with M\"obius band base and one exceptional fiber}

In this short section we prove the following:

\begin{proposition}\label{prop:f1,1case}
Let $M$ be a Seifert manifold with base the M\"obius band $F_{1,1}$ and one exceptional fiber; then there exists a non-horizontal, non-vertical curve $c\subseteq \partial M$ such that $\Sk(M,\Q(A))$ is finitely generated as a $\Q(A)[c]$-module.
\end{proposition} 

The proof relies on the following simple lemma:
\begin{lemma}
	\label{lemma:finitenessDehnSurgery} Assume that $M$ is a compact oriented $3$-manifold with boundary $\partial M=T \cup \Sigma,$ where $T$ is a $2$-torus, and that $\Sk(M,\Q(A))$ is finitely generated $\Q(A)[c_1,\ldots,c_n]$-module, for some disjoint simple closed curves $c_1,\ldots,c_n$ on $\partial M,$ with $c_1$ a curve on $T.$
	
	If $M'$ is obtained from $M$ by Dehn surgery on $T$ of slope $c_1,$ then $\Sk(M',\Q(A))$ is a finitely generated $\Q(A)[c_2,\ldots,c_n]$-module
\end{lemma}
\begin{proof}
	It is well known that links in $M'$ can always be isotoped to lie in $M,$ and therefore $\Sk(M',\Q(A))$ is spanned by links in $M$ and is a quotient of $\Sk(M,\Q(A)).$ Moreover, since $c_1$ bounds a disk in $M',$ for any $z\in \Sk(M',\Q(A))$ one has the relation $$c_1\cdot z=(-A^2-A^{-2})z.$$  Since $\Sk(M,\Q(A))$ is a finitely generated $\Q(A)[c_1,\ldots,c_n]$-module, its quotient by those relations is a finitely generated $\Q(A)[c_2,\ldots,c_n]$-module, and thus so is $\Sk(M',\Q(A)).$
\end{proof}
\begin{proof}[Proof of Proposition \ref{prop:f1,1case}]
	If $M$ is as in the hypothesis, it can be obtained by Dehn surgery on $\F$ along some non-horizontal,non-vertical curve $c'\subseteq \partial F$. Then by Proposition \ref{prop:f1,2case} there must be $c\subseteq \partial F$ such that $\Sk(\F,\Q(A))$ is finitely generated as a $\Q(A)[c,c']$-module. Then Lemma \ref{lemma:finitenessDehnSurgery} gives the desired result.
\end{proof}
\section{Torsion in the skein modules of Seifert manifolds}
\label{sec:Seifert}
In this section, we will make use of the various criteria of the previous sections, namely Theorems \ref{thm:infiniteX(M)}, \ref{thm:largeX(M)_nclosed}, \ref{thm:torsiontorus} and \ref{thm:nsep-torus}, to show Theorem \ref{thm:SFStorsion} of the introduction, which constitutes a verification of Conjecture \ref{conj:torsion} for all orientable Seifert manifolds.

Whenever possible, we will rely on the latter two criterions, Theorem \ref{thm:torsiontorus} and \ref{thm:nsep-torus}, as they provide explicit torsion elements.

To begin with, let us first recall the presentation of $\pi_1(M)$ when $M$ is a Seifert fibert manifold (possibly with boundary). As is customary, let us denote with $M(g,n;(\beta_1,\alpha_1),\ldots ,(\beta_k,\alpha_k))$ the Seifert manifold with base a surface of genus $g$ (with the convention that $g<0$ means that the base is not orientable) and $n$ boundary components, and with $k$ exceptional fibers of parameters $(\alpha_i,\beta_i).$
\begin{theorem}
\label{thm:SFSpi1}\cite[Theorem 6.1]{JN83}
Let $M=M(g,n;(\beta_1,\alpha_1),\ldots ,(\beta_k,\alpha_k))$ be a Seifert manifold. If $g\geq 0,$ then

\begin{multline}\label{eq:presSFSor}\pi_1(M)=\langle a_1,b_1,\ldots ,a_g,b_g,q_1,\ldots ,q_k,c_1,\ldots,c_n,h| 
	 [h,a_i]=[h,b_i]=1, \\ [h,c_j]=[h,q_l]=1,  q_l^{\alpha_l}h^{\beta_l}=1, q_1\ldots q_k c_1\ldots c_n[a_1,b_1]\ldots [a_g,b_g]=1 \rangle,
\end{multline}

and if $g<0$ then
\begin{multline}\label{eq:presSFSnonor}\pi_1(M)=\langle a_1,\ldots,a_{|g|},q_1,\ldots ,q_k,c_1,\ldots,c_n,h| 
a_iha_i^{-1}=h^{-1},[h,c_j]=[h,q_l]=1, \\ q_l^{\alpha_l}h^{\beta_l}=1, q_1\ldots q_k c_1\ldots c_n a_1^2 \ldots a_{|g|}^2=1 \rangle,
\end{multline}

\end{theorem}
To prove Theorem \ref{thm:SFStorsion}, we will need a few lemmas:
\begin{lemma}
	\label{lemma:positive_genus} Let $M=M(g,n;(\beta_1,\alpha_1),\ldots ,(\beta_k,\alpha_k))$ and assume $g>0,$ or $g<-1$ Then $\mathcal{S}(M)$ has $(A\pm 1)$-torsion.
\end{lemma}
\begin{proof}
	If $g>0$, a simple closed curve on the base of $M$ representing the generator $a_1$ gives rise to a vertical torus which is non-separating and not $2$-torsion in $H_1(M,\Z).$ Hence Theorem \ref{thm:nsep-torus} applies.
	Similarly, if $g<-1,$ a simple closed curve on the base $F$ of $M$ representing $a_1$ is non-zero in $H_1(F,\Q),$ hence non-separating in $F.$ It gives rise to a non-separating torus in $M$ which is not $2$-torsion in $H_1(M,\Z),$ so Theorem \ref{thm:nsep-torus} applies again.
\end{proof}
\begin{lemma}
	\label{lemma:SFSor} Let $M=M(0,n;(\beta_1,\alpha_1),\ldots ,(\beta_k,\alpha_k))$ and assume $n+k\geq 4.$ Then $\mathcal{S}(M)$ has $(A\pm 1)$-torsion.
\end{lemma}
\begin{proof}
	Let $\delta$ be a simple closed curve on the base of $M,$ representing either $c_1c_2,$ $q_1c_1$ or $q_1q_2$ in $\pi_1(M)$ (depending whether $n\geq 2,$ $n=1$ or $n=0$). The simple closed curve $\delta$ gives rise to a separating vertical torus $T$ in $M.$ 
	We will construct a representation $\rho:\pi_1(M)\longrightarrow \slC$ that satisfies the hypothesis of Theorem \ref{thm:torsiontorus}.
	
	We note that, without loss of generality, we can assume $n=0.$ Indeed, if $n\geq 1,$ performing a Dehn-surgery along the boundary components of $M,$ one may obtain a closed Seifert manifold $M'$ over $S^2$ with at least $4$ exceptional fibers, and any $\slC$ representation of $\pi_1(M')$ that satisfies the hypothesis of Theorem \ref{thm:torsiontorus} with respect to the torus $T$ will also give a suitable representation of $\pi_1(M).$
	
	Therefore, we assume $n=0,k\geq 4.$ Note that the two component $M_1$ and $M_2$ of $M\setminus T$ are again Seifert fibered, with $M_1$ having base a disk with $2$ cone points, and $M_2$ having base a disk with $k-2$ cone points.
	
	We construct a representation $\rho_1$ of $\pi_1(M_1)$ by $\rho_1(h)=-I_2,$ $\rho_1(q_i)$ of order $d_i=\alpha_i$ (if $\beta_i$ is even) or $d_i=2\alpha_i$ (if $\beta_i$ is odd) and $\rho_1(q_1q_2)$ of trace $t$ where $t\neq \pm 2\in \C$ is a parameter that we will fix later. Note that $q_1$ and $q_2$ have order at least $3$ since if $\alpha_i=2$ then $\beta_i$ is odd. Asking $\rho_1(q_i)$ of order $d_i>2$ can be achieved by asking $\tr(\rho_1(q_i))=\zeta+\zeta^{-1}$ where $\zeta$ is a primitive root of order $d_i.$
	We claim that there is such a representation, since for a free group $F_2=\langle a,b\rangle$ there is a representation achieving $\tr(a)=x,\tr(b)=y,\tr(ab)=z$ for any $x,y,z\in \C.$ Moreover, this representation is irreducible for almost all values of $t.$
	
	Next we construct similarly a representation $\rho_2$ of $\pi_1(M_2)$ asking that $\rho_2(h)=-1,\rho_2(q_i)=\pm I_2$ for $i\geq 5$ (the sign being determined by the parities of $\alpha_i$ and $\beta_i$) and $\tr(\rho_2(q_i))=\zeta_i+\zeta_i^{-1}$ for $i=3,4,$ where $\zeta_i$ has order $d_i=\alpha_i$ or $2\alpha_i$ as before.
	
	We have $\rho_2(q_3\ldots q_k)=\varepsilon \rho_2(q_3q_4)$ for some sign $\varepsilon \in \lbrace \pm 1 \rbrace.$ The representations $\rho_1$ and $\rho_2$ induce a representation $\rho$ of $\pi_1(M)$ if and only if
	$$\rho_1(q_1q_2)=\rho_2(q_3\ldots q_k)^{-1},$$
	which will be the case up to conjugating $\rho_2$ if $\tr(\rho(q_3q_4))=\varepsilon t,$ which can again be realized for some choice of $\rho_2(q_3),\rho_2(q_4).$ 
	
	Moreover, for most values of $t,$ the induced representation $\rho$ will be non-central on $T$ (since $t\neq \pm 2,$) and irreducible on $\pi_1(M_i)$ for $i=1$ or $2,$ therefore it satisfies the criterion of Theorem \ref{thm:torsiontorus}, and $\mathcal{S}(M)$ will have $(A\pm 1)$-torsion.

\end{proof}
\begin{lemma}
	\label{lemma:SFSnor} Let $M=M(-1,n;(\beta_1,\alpha_1),\ldots ,(\beta_k,\alpha_k))$ where $n+k\geq 3.$ Then $\mathcal{S}(M)$ has $(A\pm 1)$-torsion.
\end{lemma}
\begin{proof}
The proof is similar to the proof of Lemma \ref{lemma:SFSor}. We take $\delta$ to be a simple closed curve on the base of $M,$ representing $c_1c_2$, $c_1q_1$ or $q_1q_2,$ (depending whether $n\geq 2,$ $n=1$ or $n=0$). Its fiber is again a separating torus $T$ in $M,$ since this loop is orientation preserving on the base. We construct a representation $\rho:\pi_1(M)\rightarrow \slC$ satisfying Theorem \ref{thm:torsiontorus}, and again without loss of generality we can assume $n=0.$

The two components of $M\setminus T$ are then $M_1$ and $M_2,$ with $M_1$ Seifert fibered over a disk with $2$ cone points, and $M_2$ Seifert fibered over a Möbius band with $k-2$ cone points.

We construct a representation $\rho_1$ of $\pi_1(M_1)$ as before, taking $\rho_1(h)=-I_2,$ and $\rho_1(q_i)$ of order $d_i=\alpha_i$ or $2\alpha_i$ as in the proof of Lemma \ref{lemma:SFSor}, and $\rho_1(q_1q_2)$ of trace $t$ chosen so that $\rho_1$ is irreducible.

Then, we choose $\rho_2$ so that $\rho_2(h)=-I_2, \rho_2(q_1q_2)=\rho_1(q_1q_2)$ and  $\rho_2(q_i)=\pm I_2$ for $i\geq 4,$ the sign being determined by the parities of $\alpha_i,\beta_i,$ and $\rho_2(q_3)$ of order $d_3=\alpha_3$ or $2\alpha_3$ depending on the parities of $\alpha_3,\beta_3,$ and we have again that $d_3$ is at least $3.$ To finish, we just need to take $\rho_2(a_1)$ to be a square root of $\rho_2(q_1q_2q_3\ldots q_k)^{-1}.$ This is possible as square roots always exist in $\slC.$ 
We will have that $\rho_2$ will also be irreducible for most choices of $t=\Tr \rho_1(c_1c_2),$ and Theorem \ref{thm:torsiontorus} applies.
	
\end{proof}
The next lemma complements Lemma \ref{lemma:SFSnor}, but is treated separately since we will rely on Theorem \ref{thm:largeX(M)_nclosed} instead of Theorem \ref{thm:nsep-torus} to find torsion:
\begin{lemma}
	\label{lemma:SFSnor2} Let $M=M(-1,n;(\beta_1,\alpha_1),\ldots ,(\beta_k,\alpha_k))$ where $n+k=2.$ Then $\mathcal{S}(M)$ has $(A\pm 1)$-torsion.
\end{lemma}
\begin{proof}
If $k=2$ and $n=0,$ then $M$ is a closed Haken manifold and $X(M)$ is infinite by \cite[Proposition 3.1]{DKS2}, hence $\Sk(M)$ has $A\pm 1$ torsion. Thus we restrict to the cases $n=1,k=1$ or $n=2,k=0.$

First we treat the case $n=2,k=0.$ Let $c_1,c_2$ be two simple closed curves on the boundary components of $M$ such that $\Sk(M,\Q(A))$ is a finitely generated $\Q(A)[c_1,c_2]$-module. Those curves are provided by Proposition \ref{prop:f1,2case}, and they are non-vertical. We claim that $t_{c_1}$ and $t_{c_2}$ are algebraically independent in $\C[X(M)].$ Indeed, consider the base $B\simeq F_{-1,2}$ of $M,$ and notice that $\pi_1(M)$ surjects onto $\pi_1(B)$ under the canonical projection. The restrictions of $t_{c_1}$ and $t_{c_2}$ to $X(B)$ are polynomials in $t_a$ and $t_b$ respectively, where $a$ and $b$ are the boundary components of $B.$ Since $\pi_1(B)\simeq F_2,$ and $X(B)=\C[t_a,t_b,t_{ab}],$ we get that $X(B)$ has dimension $3$ and thus $X(M)$ has dimension at least $3.$ Moreover, we get that $t_a$ and $t_b$ are algebraically independent in $\C[X(B)],$ and thus $t_{c_1}$ and $t_{c_2}$ are algebraically independent in $\C[X(M)].$ Therefore we can apply Theorem \ref{thm:largeX(M)_nclosed}, and $\Sk(M)$ has $A\pm 1$ torsion.

Finally, when $n=1$ and $k=1,$ let $(\beta,\alpha)$ be the parameters of the exceptional fiber. By Proposition \ref{prop:f1,1case}, there is a non-vertical simple closed curve $c_1$ on the boundary component of $M,$ such that $\Sk(M,\Q(A))$ is a finitely generated $\Q(A)[c_1]$-module. Notice that $c_1$ is non-trivial in $H_1(M,\Q),$ which implies (looking at abelian characters) that $t_{c_1}$ takes infinitely many distinct values on $X(M).$ This implies that $t_{c_1}$ is transcendental in $\C[X(M)].$ 

Next we recall that, by Theorem \ref{thm:SFSpi1}, we have that \begin{equation}\label{eq:presSFS2}\pi_1(M)=\langle a,q,c,h | aha^{-1}=h^{-1}, [h,c]=[h,q]=1, q^{\alpha}h^{\beta}=1, qca^2=1\rangle.
\end{equation}
Let us show that for any $x,y\in \C,$ there is a representation $\rho:\pi_1(M)\longrightarrow \slC$ such that $\rho(h)=-I_2,$ and $\tr (\rho(c))=x$ and $\tr(\rho(qc))=y,$ which implies that $\dim X(M)\geq 2.$ Note that if $\rho(h)=-I_2$ then the first $3$ relations of the presentation are satisfied. We define $\rho$ on $\langle q,c \rangle$ so that $\tr \rho(q)=\zeta+\zeta^{-1}$ where $\zeta$ is a primitive root of unity of order $\alpha$ if $\beta$ is even, and $2\alpha$ if $\beta$ is odd, and $\tr \rho(c)=x$ and $\tr \rho (q c)=y.$ It is always possible to realize this, since the character variety of a free group $F_2=\langle u,v \rangle$ is $\C[\tr u,\tr v, \tr uv].$ Moreover, note that $\zeta$ has order at least $3$ for any coprime $(\alpha,\beta)$ with $\alpha\geq 2.$ Hence $\zeta \neq \zeta^{-1},$ and the fact that $\tr \rho (q)=\zeta+\zeta^{-1}$ implies that $\rho(q)$ is diagonalizable and the fourth relation of Eq \ref{eq:presSFS2} is verified. Finally, since square roots always exist in $\slC,$ we can choose $\rho(a)$ so that the last equation is verified.

Therefore $\dim X(M)\geq 2$ and we can apply Theorem \ref{thm:largeX(M)_nclosed} again and $\Sk(M)$ has $A\pm 1$ torsion.
\end{proof}
We are now ready to complete the:
\begin{proof}[Proof of Theorem \ref{thm:SFStorsion}]
We note that by \cite[Theorem 1.3]{DKS2}, if $M$ is closed and Haken then $X(M)$ has positive dimension, which by Theorem \ref{thm:infiniteX(M)} implies that $\Sk(M)$ has $A\pm 1$ torsion. We deal with the case of Seifert manifolds with boundary.

Let $M$ be a Seifert manifold with base $B$ and with boundary. It is well-known (see \cite[Proposition 1.12]{Hatcher:notes} for example) that an incompressible surface $S$ in $M$ can be isotoped to be either horizontal (meaning the projection $\pi:M\rightarrow B$ is a local homeomorphism) or vertical (meaning it consists of fibers of $\pi$). In the former case, $S$ can not be closed when $M$ has boundary. In the latter case, if $S$ is closed it must be a vertical torus, and for $S$ to be non-boundary parallel, there must be a simple closed curve in $B$ which is non-boundary parallel and does not bound a disk with at most one cone point on either side. This is possible exactly when $M$ fibers on neither of
\begin{itemize}
	\item[-]a disk with at most two exceptional fibers or
	\item[-]an annulus  or Möbius band with at most one exceptional fiber.
\end{itemize}
Then $M$ has to be in one of the cases covered by Lemma \ref{lemma:positive_genus}, \ref{lemma:SFSor}, \ref{lemma:SFSnor} or \ref{lemma:SFSnor2}, and $\Sk(M)$ has $A\pm 1$ torsion.  
\end{proof}
\section{Torsion detected by the skein module at $A=\sqrt{-1}$}\label{sec:rp3}
In this section we give examples of manifolds where torsion in $\Sk(M)$ is detected not by the character variety, but rather by the computations of $\Sk_{\sqrt{-1}}(M).$

\subsection{The skein module of $L(2p,1)\#L(2,1)$}

In this subsection we provide a presentation for $\Sk_{\sqrt{-1}}\left(L(p,1)\#L(2,1)\right)$ as an explicit quotient of $\Sk_{\sqrt{-1}}(H_2)=\C[x,y,z]$, where the skein module of the genus $2$-handlebody is the polynomial algebra over the three curves $x,y,z$ of Figure \ref{fig:g2curves}.
\begin{figure}
	\centering
	\includegraphics[height=3cm, width=0.5\textwidth]{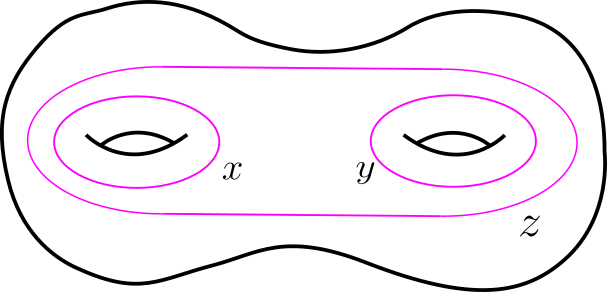}
	\caption{The curves $x,y,z$ which (with their parallels) generate $\Sk(H_2)$}\label{fig:g2curves}
\end{figure}
The calculations rely on two main ingredients: the observation that at a fourth root of unity, switching a crossing is the same as changing the sign of the skein element; and the handle slide presentation for skein modules, which relies on the following result of Hoste and Przytycki.

\begin{theorem}\label{thm:handleslide}\cite{HP93}
	Let $M$ be a $3$-manifold with boundary, let $\alpha$ be a curve in $\partial M$ and let $M'$ be the manifold obtained from  $M$ by attaching a handle to $\alpha$. Then $\Sk(M')=\Sk(M)/J$, where $J$ is the submodule of $\Sk(M)$ generated by elements of the form $\gamma-sl_{\alpha}\gamma$, where $\gamma$ is a framed link is $M$ and $sl_{\alpha}\gamma$ is the result of handle-sliding a component of $\gamma$ over $\alpha$ (see Figure \ref{fig:handleslide}).
\end{theorem}
\begin{figure}
\centering
\begin{minipage}{.45\textwidth}
	\centering    
	\includegraphics[width=0.6\textwidth]{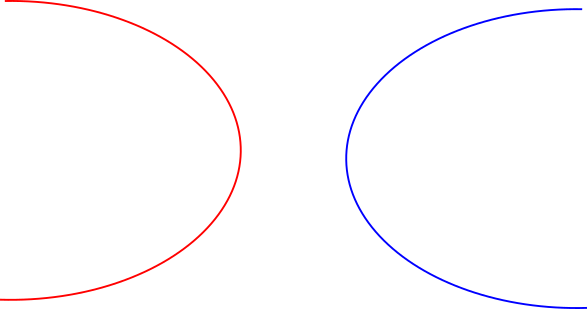}
\end{minipage}  $\longrightarrow$
\begin{minipage}{.45\textwidth}
	\centering    
	\includegraphics[width=0.5\textwidth]{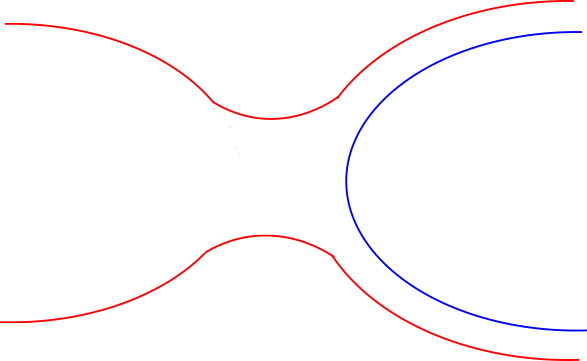}
\end{minipage}
\caption{Handlesliding the red curve over the blue curve}\label{fig:handleslide}
\end{figure}

Notice that when handle-sliding a curve $\gamma$ over a curve $\alpha$, the result will not only depend on $\gamma$ and $\alpha$ but also on an arc used to connect $\gamma$ and $\alpha$; different choices of arc will usually lead to different curves. The first step needed for the calculations is to reduce the number of handle slide that need to be considered. This is contained in the following Lemma.

\begin{figure}
	\centering
	\begin{minipage}{.45\textwidth}
		\centering    
		\includegraphics[width=0.9\textwidth]{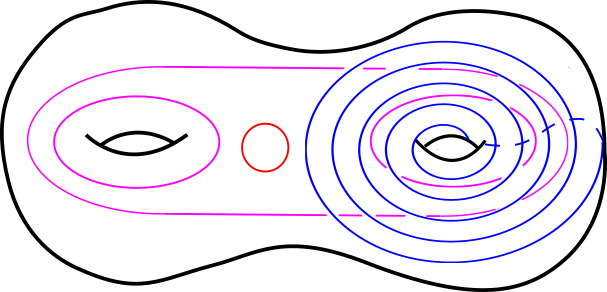}
\label{fig:g2hs1before}
\vspace{0.35cm}
	\end{minipage}  $\longrightarrow$
	\begin{minipage}{.45\textwidth}
		\centering    
		\includegraphics[width=0.9\textwidth]{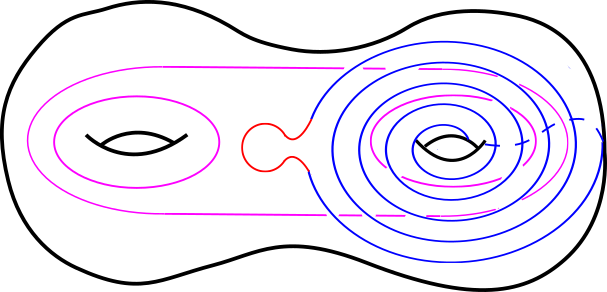}
\label{fig:g2hs1after}
\vspace{0.35cm}
	\end{minipage}
	\vspace{0.35cm}
\begin{minipage}{.45\textwidth}
	\centering    
	\includegraphics[width=0.9\textwidth]{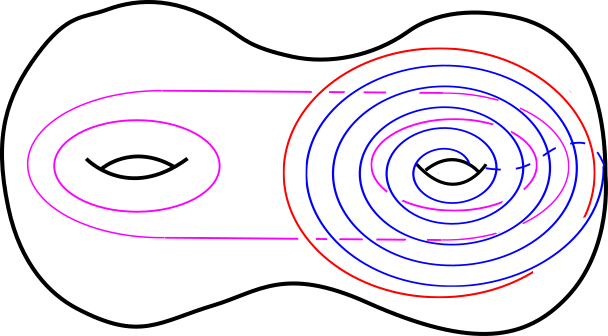}
	\label{fig:g2hs3before}
\end{minipage}  $\longrightarrow$
\begin{minipage}{.45\textwidth}
	\centering    
	\includegraphics[width=0.9\textwidth]{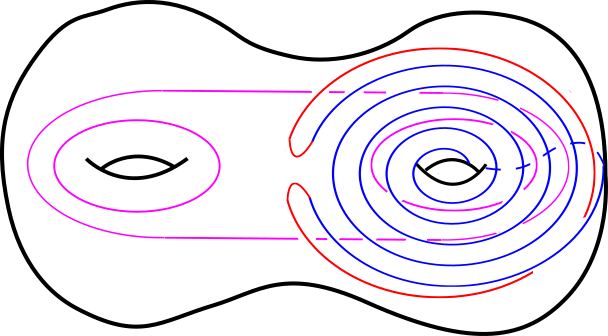}
	\label{fig:g2hs3after}
\end{minipage}
\vspace{0.35cm}
\begin{minipage}{.45\textwidth}
	\centering    
	\includegraphics[width=0.9\textwidth]{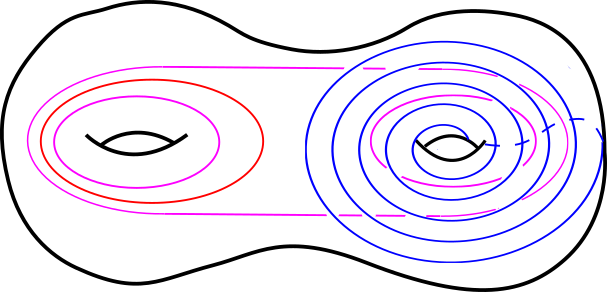}
	\label{fig:g2hs2before}
\end{minipage}  $\longrightarrow$
\begin{minipage}{.45\textwidth}
	\centering    
	\includegraphics[width=0.9\textwidth]{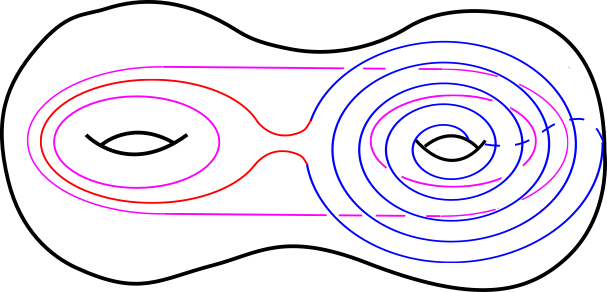}
	\label{fig:g2hs2after}
\end{minipage}

\begin{minipage}{.45\textwidth}
	\centering    
	\includegraphics[width=0.9\textwidth]{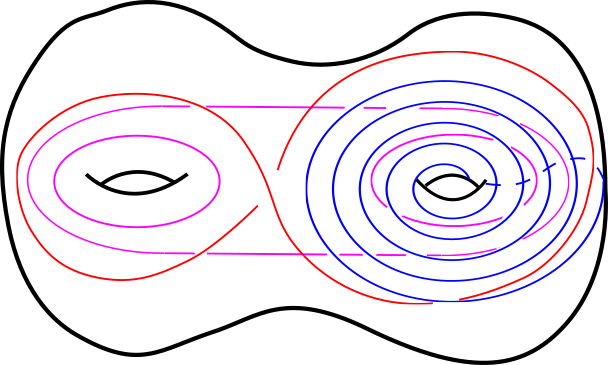}
	\label{fig:g2hs4before}
\end{minipage}  $\longrightarrow$
\begin{minipage}{.45\textwidth}
	\centering    
	\includegraphics[width=0.9\textwidth]{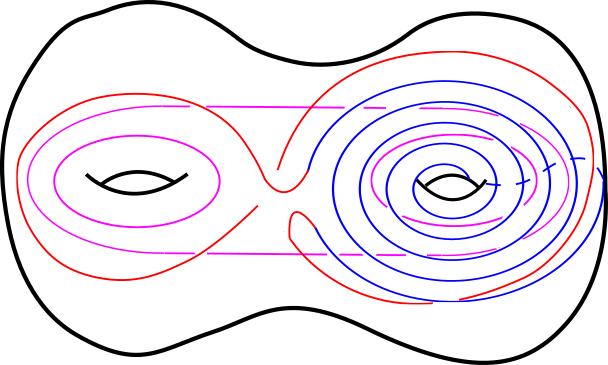}
	\label{fig:g2hs4after}
\end{minipage}
\caption{Handleslides over $\alpha_2$}	\label{fig:g2handleslides}
\end{figure}
\begin{lemma}
	Let $M$ be a $3$-manifold with a genus $2$ Heegaard splitting, call $H_2$ the first handlebody in the decomposition and $\alpha_1$ and $\alpha_2$ two non-parallel, disjoint curves in $H_2$ that bound disks in the other handlebody. Then $\Sk(M)=\Sk(H_2)/J$, where $J$ is the submodule generated by the handleslides pictured in Figure \ref{fig:g2handleslides} (over the curve $\alpha_2$) and the symmetric handleslides over the curve $\alpha_1$.
\end{lemma}
\begin{proof}
	We show that $J$ is actually the submodule generated by all handle-slide relations. To do so, we take any curve $\gamma\subseteq H_2$, connect it with an arbitrary arc to $\alpha_2$ and handle-slide it over; call the resulting curve $sl_{\alpha_2}\gamma$. Take a collar $C$ of the boundary of $H_2$; isotope $\gamma$ along the arc so that it is close to $\alpha_2$ and so that it intersects $\partial C$ in two points and call $\overline{\gamma}=\gamma\cap H_2\setminus \mathring{C}$. After the slide, the part of $sl_{\alpha_2}\gamma$ coming from $\alpha_2$ can be pushed into $C$ so that  $sl_{\alpha_2}\gamma$ intersects $\partial C$ in two points; call $\overline{sl_{\alpha_2}\gamma}=sl_{\alpha_2}\gamma\cap H_2\setminus \mathring C$. Because of \cite[Lemma 5.2]{Le06}, the relative skein module of $H_2$ with two points on the boundary is generated by the elements of Figure \ref{fig:g2rs} (where the additional closed curves can be comprised of multiple parallel copies), so $\overline{sl_{\alpha_2}\gamma}$ can be modified, using skein relations and isotopies that do not move its endpoints, to be one of the cases of Figure \ref{fig:g2handleslides}.
	
\end{proof}
\begin{figure}
	\centering
	\begin{minipage}{.45\textwidth}
		\centering    
		\includegraphics[width=0.9\textwidth]{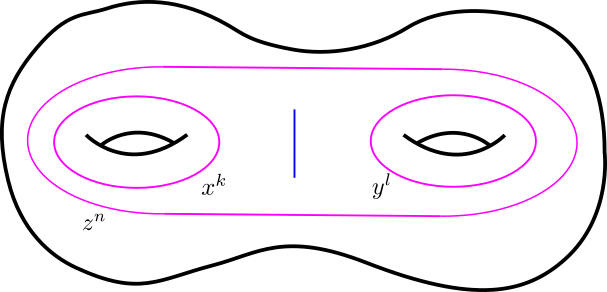}
		\vspace{0.5cm}
	\end{minipage} 
	\begin{minipage}{.45\textwidth}
		\centering    
		\includegraphics[width=0.9\textwidth]{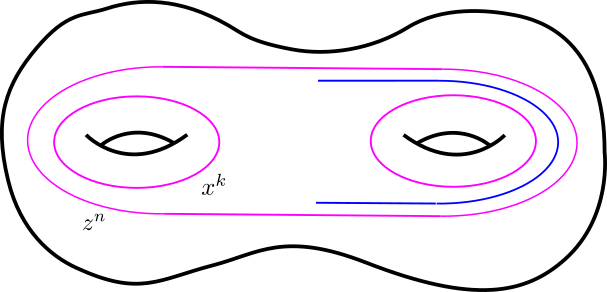}
		\vspace{0.5cm}
	\end{minipage}

	\centering
	\begin{minipage}{.45\textwidth}
		\centering    
		\includegraphics[width=0.9\textwidth]{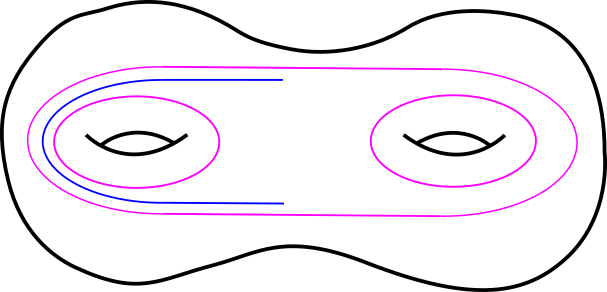}
	\end{minipage}  
	\begin{minipage}{.45\textwidth}
		\centering    
		\includegraphics[width=0.9\textwidth]{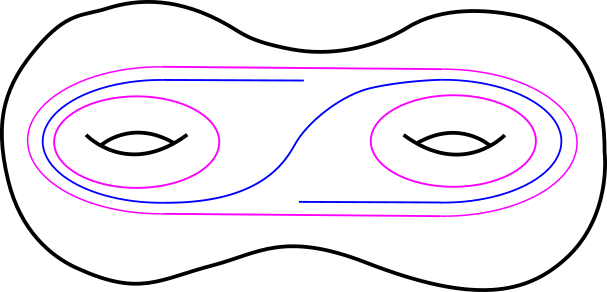}
	\end{minipage}
	\caption{}\label{fig:g2rs}
\end{figure}

Recall the two families of Chebichev polynomials $T_n$ and $S_n$, defined by the following recurrence relations:
\begin{equation*}
	\begin{cases}
		T_0(x)=2;\\
		T_1(x)=x;\\
		xT_n(x)=T_{n+1}+T_{n-1}
	\end{cases}
\end{equation*}
and
\begin{equation*}
\begin{cases}
	S_0(x)=1;\\
	S_1(x)=x;\\
	xS_n(x)=S_{n+1}+S_{n-1}
\end{cases}
\end{equation*}
\begin{proposition}\label{prop:presentationlp1l21}
	The skein module at $\sqrt{-1}$ of $L(2,1)\#L(p,1)$ is $\C[x,y,z]/J$, where $J$ is the submodule generated by the following elements:
	\begin{enumerate}
		\item 
		\begin{equation*}
			\begin{cases}
				x^ky^lz^n\left(2-\sqrt{-1}^pT_p(y)\right) \text{ if }l+n\text{ is even;}\\
				x^ky^lz^n\left(2+\sqrt{-1}^pT_p(y)\right) \text{ if }l+n\text{ is odd;}
			\end{cases}
		\end{equation*}
		
		\item $$
		\begin{cases}
			x^ky^lz^n\left(y+\sqrt{-1}^pT_{p-1}(y)\right) \text{ if }l+n\text{ is even;}\\
			x^ky^lz^n\left(y-\sqrt{-1}^pT_{p-1}(y)\right) \text{ if }l+n\text{ is odd;}
		\end{cases}$$
		\item $$
		\begin{cases}
			x^ky^lz^n\left(x-\sqrt{-1}^p\left(z S_{p-1}(y)-xS_{p-2}(y)\right)\right) \text{ if }l+n\text{ is even;}\\
			x^ky^lz^n\left(x+\sqrt{-1}^p\left(z S_{p-1}(y)-xS_{p-2}(y)\right)\right) \text{ if }l+n\text{ is odd;}
		\end{cases}$$
		\item $$
		\begin{cases}
			x^ky^lz^n\left(\sqrt{-1}z-\sqrt{-1}xy+\sqrt{-1}^{p-1}\left(zS_{p-2}(y)-xS_{p-3}(y)\right)\right) \text{ if }l+n\text{ is even;}\\
			x^ky^lz^n\left(\sqrt{-1}z-\sqrt{-1}xy-\sqrt{-1}^{p-1}\left(zS_{p-2}(y)-xS_{p-3}(y)\right)\right)  \text{ if }l+n\text{ is odd;}
		\end{cases}$$
		\item $$
		\begin{cases}
			x^ky^lz^n\left(x^2\right) \text{ if }k+n\text{ is even;}\\
			x^ky^lz^n\left(4-x^2\right) \text{ if }k+n\text{ is odd;}
		\end{cases}$$
		\item $$
		x^ky^lz^n\left(2x\right) \text{ if }k+n\text{ is odd;}
		$$
		\item $$
		\begin{cases}
			x^ky^lz^n\left(zx\right) \text{ if }k+n\text{ is even;}\\
			x^ky^lz^n\left(2y-z x\right) \text{ if }k+n\text{ is odd;}
		\end{cases}$$
		\item $$
		\begin{cases}
			x^ky^lz^n\left(2\sqrt{-1}z-\sqrt{-1}xy\right) \text{ if }l+n\text{ is even;}\\
			x^ky^lz^n\left(\sqrt{-1}xy\right)  \text{ if }l+n\text{ is odd;}
		\end{cases}$$
	\end{enumerate}
\end{proposition}
\begin{proof}
	\begin{figure}
		\centering
		\includegraphics[height=3cm, width=0.45\textwidth]{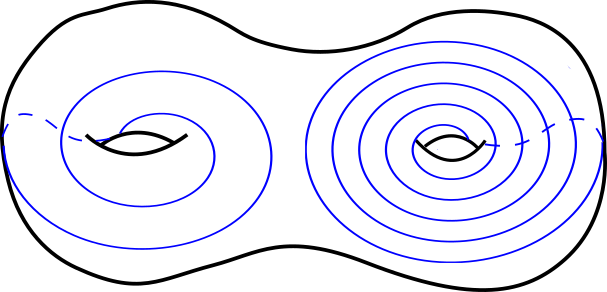}	
		\caption{Handle attachments for $L(p,1)\#\mathbb{RP}^3$}\label{fig:g2handles}	
\end{figure}	
	Take a handlebody $H_2$ of genus $2$, and take the two curves $
	\alpha_1$ and $\alpha_2$ depicted in Figure \ref{fig:g2handles}; attaching handles along those curves (and filling the resulting spherical boundary) gives $L(2,1)\#L(p,1)$. Therefore we can obtain all relations for $\Sk_{\sqrt{-1}}(L(2,1)\#L(p,1))$ by performing the handleslides depicted in Figure \ref{fig:g2handleslides} along either $\alpha_1$ or $\alpha_2$. We perform the calculations for $\alpha_2$; the calculations for $\alpha_1$ are obtained by switching $x$ and $y$ and setting $p=2$.
	
	Consider the first handle-slide relation depicted in Figure \ref{fig:g2handleslides}. The left hand side is equal to $2x^ky^lz^n$ (remember that at $A=\sqrt{-1}$, $A^2+A^{-2}=-2$), whereas the right hand side is equal to $(-1)^{l+n}(-A^3)x^ky^lz^n\gamma_p$, where $\gamma_p$ is the curve of Figure \ref{fig:d2gammaf1}, sitting in a collar of $\partial H_2$ far away from the rest of the link. To compute the right hand side, pull the handle-slid curve through the rest of $x^ky^lz^n$; this introduces a sign of $(-1)^{l+n}$ and a positive framing. Then the following lemma about $\gamma_p$ provides the first set of relations:
	\begin{figure}
		\centering
		\includegraphics[height=3cm, width=0.45\textwidth]{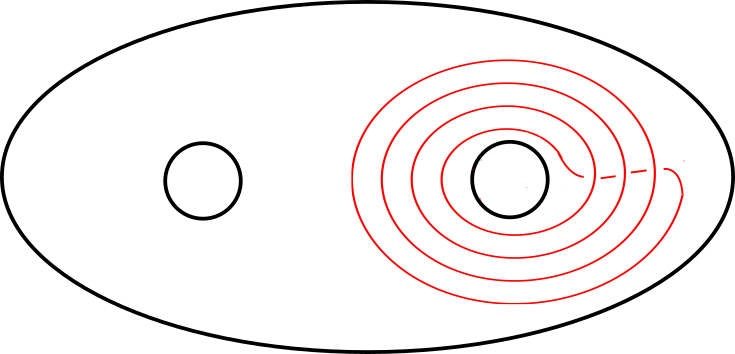}
		\caption{The curve $\gamma_p$ winds $p$ times around the puncture}\label{fig:d2gammaf1}
	\end{figure}
	\begin{lemma}\label{lem:gammap}
		The curve $\gamma_p$ of Figure \ref{fig:d2gammaf1} satisfies the following:
		\begin{enumerate}
		\item $\gamma_1=y$;
		\item $\gamma_2=Ay^2-A-A^{-3}$;
		\item $\gamma_p=Ay\gamma_{p-1}-A^2\gamma_{p-2}$ for all $p\geqslant 3$;
		\item when $A=\sqrt{-1}$, $\gamma_p=A^{p-1} T_p(y)$ for all $p\geqslant 1$.
		\end{enumerate}
	\end{lemma}
	\begin{proof}[Proof of Lemma \ref{lem:gammap}]
		The first two properties are directly computed. 
		The third property comes from the following skein computation:
		\begin{align*}
			\vcenter{\hbox{\includegraphics[width=0.25\textwidth]{figures/d2gammaf1}}}&=A\hspace{1mm}\vcenter{\hbox{\includegraphics[width=0.25\textwidth]{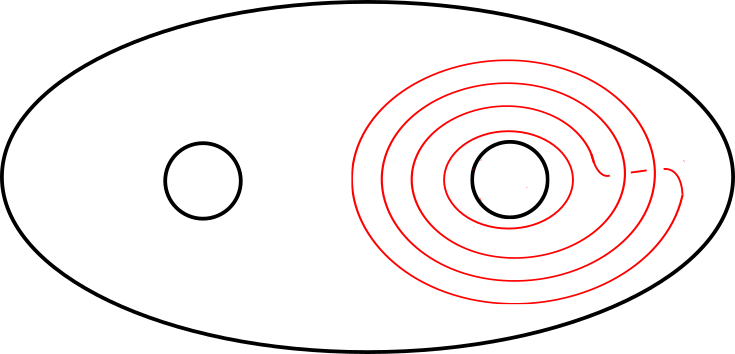}}}+A^{-1}\hspace{1mm}\vcenter{\hbox{\includegraphics[width=0.25\textwidth]{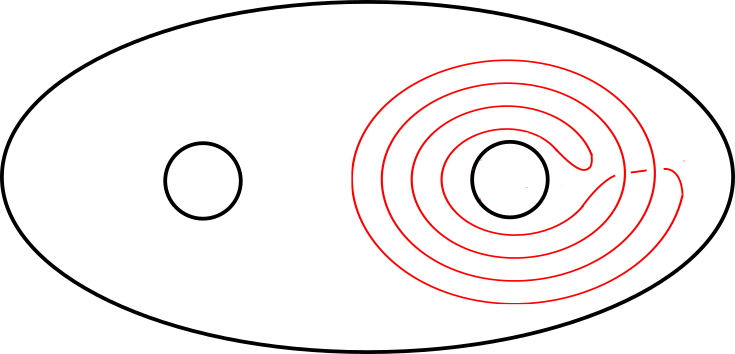}
				}}=\\ &=Ay\gamma_{p-1}-A^2\gamma_{p-2}
		\end{align*}	
	where the curve depicted in the left hand side winds $p$ times around the puncture.
	
	For the last part, define $Q_p=A^{-p+1}\gamma_p$; this new sequence satisfies the same recurrence relation as $T_p(y)$: $Q_p=yQ_{p-1}+Q_{p-2}$. Furthermore, $Q_1=y$ and $Q_2=y^2-1-A^{-4}$ which, when $A=\sqrt{-1}$, is equal to $T_2(y)$.
	\end{proof}

Now consider the second handleslide depicted in Figure \ref{fig:g2handleslides}; the left hand side is simply $x^ky^{l+1}z^n$, whereas the right hand side is, using the same trick as before, equal to $(-1)^{l+n}x^ky^lz^n\gamma_{p-1}$ (notice that there are two opposite framings canceling each other out). Once again Lemma \ref{lem:gammap} provides the next set of relations.

The third handle-slide depicted in Figure \ref{fig:g2handleslides} introduces the curve $\gamma'_p$ of Figure \ref{fig:d2gamma'f1}; the next lemma is the analog of Lemma \ref{lem:gammap} in this case.
\begin{figure}
	\centering
	\includegraphics[height=3cm, width=0.45\textwidth]{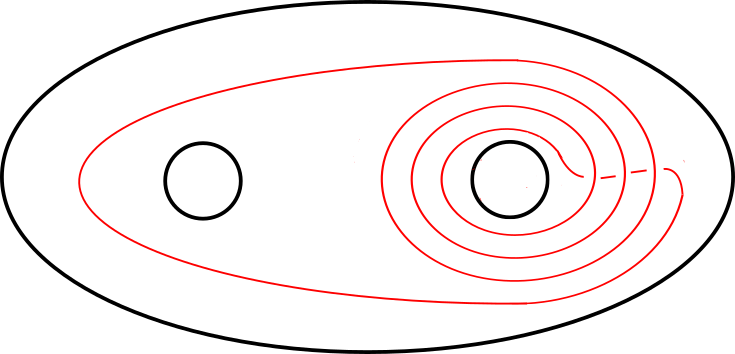}
	\caption{The curve $\gamma'_p$ winds $p$ times around the second puncture}\label{fig:d2gamma'f1}
\end{figure}

\begin{lemma}\label{lem:gamma'p}
	The curve $\gamma'_p$ of Figure \ref{fig:d2gamma'f1} satisfies the following:
	\begin{enumerate}
		\item $\gamma'_1=z$;
		\item $\gamma'_2=Ayz+A^{-1}x$;
		\item $\gamma'_p=Ay\gamma'_{p-1}-A^2\gamma'_{p-2}$ for all $p\geqslant 3$;
		\item when $A=\sqrt{-1}$, $\gamma'_p=A^p\left(A^{-1}zS_{p-1}(y)+AxS_{p-2}(y)\right)$ for all $p\geqslant 1$.
	\end{enumerate}
\end{lemma}
	\begin{proof}[Proof of Lemma \ref{lem:gamma'p}]
	The first two properties are directly computed. 
	The third property comes from the following skein computation:
	\begin{align*}
		\vcenter{\hbox{\includegraphics[width=0.25\textwidth]{figures/d2gamma_f1}}}&=A\hspace*{1mm}\vcenter{\hbox{\includegraphics[width=0.25\textwidth]{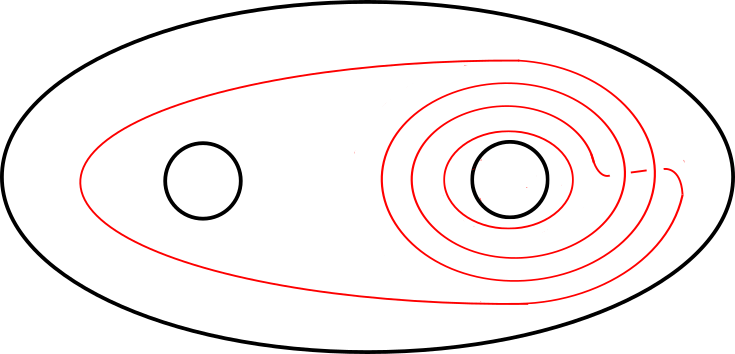}}}+A^{-1}\hspace*{1mm}\vcenter{\hbox{\includegraphics[width=0.25\textwidth]{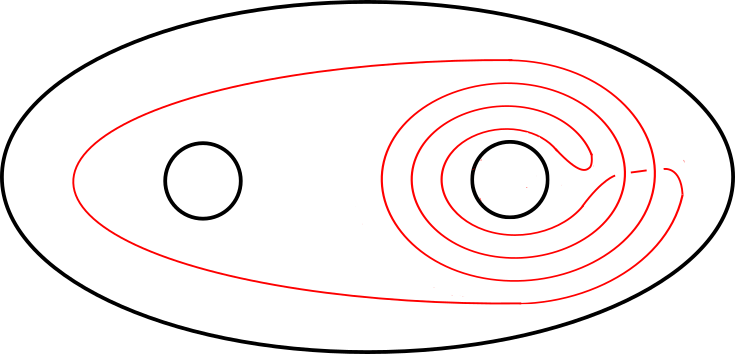}
		}}=\\ &=Ay\gamma_{p-1}-A^2\gamma_{p-2}
	\end{align*}	
	where the curve depicted in the left hand side winds $p$ times around the second puncture.
	
	For the last part, again notice that $R_p=A^{-p}\gamma'_p$ satisfies the same recurrence relation as $S_p(y)$; the difference from Lemma \ref{lem:gammap} comes from the initial conditions. In this case (when $A=\sqrt{-1}$), $R_1=-\sqrt{-1}z=-\sqrt{-1}zS_0(y)+\sqrt{-1}xS_{-1}(y)$ and $Q_2=-\sqrt{-1}yz+\sqrt{-1}x=-\sqrt{-1}zS_1(y)+\sqrt{-1}xS_0(y)$.
\end{proof}
Now if we consider the third handle-slide of Figure \ref{fig:g2handleslides}, we can see that the right hand side is equal to $(-1)^{l+n}\sqrt{-1}x^ky^lz^n\gamma'_{p}$, whereas the left hand side is equal to $x^{k+1}y^lz^n$; then Lemma \ref{lem:gamma'p} gives the next set of relations.

Finally, the left hand side of the fourth handleslide of Figure \ref{fig:g2handleslides} is equal to $\sqrt{-1}x^ky^lz^{n+1}-\sqrt{-1}x^{k+1}y^{l+1}z^n$ (after resolving the crossing), whereas the right hand side is equal to $(-1)^{l+n}\sqrt{-1}x^ky^lz^n\gamma'_{p-1}$ (this requires, as before, to pull a curve through all remaining ones, in addition to an isotopy that unwraps it once around the second puncture). This gives the next set of relations; the remaining relations come from the same argument when $p=2$ (and when $x$ and $y$ are switched).
\end{proof}

\begin{proposition}
	If $p$ is even, $\dim \Sk_{\sqrt{-1}}(L(2,1)\#L(p,1))$ is infinite.
\end{proposition}
\begin{proof}
	Consider the subspace $V$ of $\C[x,y,z]$ generated by $x^ky^lz^n$ with $k+n$ and $l+n$ even; this corresponds to the submodule $\Sk_{\sqrt{-1}}^0(L(2,1)\#L(p,1))$ of skein elements with trivial $\Z/2\Z$-homology. Notice that $V\cap J$ is generated by the elements listed in Proposition \ref{prop:presentationlp1l21} which have the appropriate parity, because the monomials that appear in each element all have the same parity with respect to both $k+n$ and $l+n$. Therefore $V\cap J$ is generated by the following elements (with the assumption that $k,l,n$ are always such that the elements belong to $V$):
	\begin{enumerate}
		
		\item $
		x^ky^lz^n\left(2-i^pT_p(y)\right)$
		
		\item $
		x^ky^lz^n\left(y-i^pT_{p-1}(y)\right);$
		\item $
		x^ky^lz^n\left(x-i^p\left(z S_{p-1}(y)-xS_{p-2}(y)\right)\right);$
		\item $
		x^ky^lz^n\left(iz-ixy-i^{p-1}\left(zS_{p-2}(y)-xS_{p-3}(y)\right)\right); $
		\item $
		x^ky^lz^n\left(x^2\right);$
		\item $
		x^ky^lz^n\left(2x\right);
		$
		\item $
		x^ky^lz^n\left(zx\right); $
		\item $
		x^ky^lz^n\left(ixy\right). $
	\end{enumerate}
	As we said, $\Sk_{\sqrt{-1}}^0(L(2,1)\#L(p,1))=V/\left(V\cap J\right)$. Denote now $J'$ the submodule of $\C[x,y,z]$ generated by monomials $x^ky^lz^n$ that satisfy:
	\begin{itemize}
		\item $k+n$ and $l+n$ even;
		\item either $k>0$ or $l>0$.
	\end{itemize}
	
	We now check that all elements in the above list belong to $J'$, (thus implying that $V\cap J\subseteq J'$). This is obvious for the latter four of the elements. The former four rely on the observation that if $p=2k$, the following hold modulo $y$: $T_p(y)\equiv(-1)^k2$, $T_{p-1}(y)\equiv 0$, $S_p(y)\equiv(-1)^k$ and $S_{p-1}(y)\equiv 0$. Therefore, all monomials in the list are divisible by either $x$ or $y$.
	
	This implies that $V/J'$ is a quotient of $V/(V\cap J)=\Sk_{\sqrt{-1}}^0(L(2,1)\#L(p,1))$; however, $V/J'$ is obviously infinite dimensional and generated by $z^{2n}$.
\end{proof}

\begin{remark}
	For the cases $p=3$ and $p=5$, it is possible to write down explicitly the relations of Proposition \ref{prop:presentationlp1l21} to find that $\Sk_{\sqrt{-1}}(L(p,1)\# \rp)$ is infinite dimensional; however we were unable to do so in the general case.
\end{remark}
\subsection{Torsion in $M\# \rp$}
	In this section we prove that $\Sk_{\sqrt{-1}}(M)$ embeds into $\Sk_{\sqrt{-1}}(M\#\mathbb{RP}^3)$; this will allow us to find torsion in a further family of examples.
	
	\begin{figure}
		\centering
		\begin{minipage}{0.45\textwidth}
			\centering
			\includegraphics[height=2.5cm, width=0.65\textwidth]{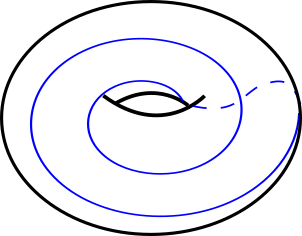}
		\end{minipage}
	\begin{minipage}{0.45\textwidth}
		\centering
	\includegraphics[height=3cm, width=0.9\textwidth]{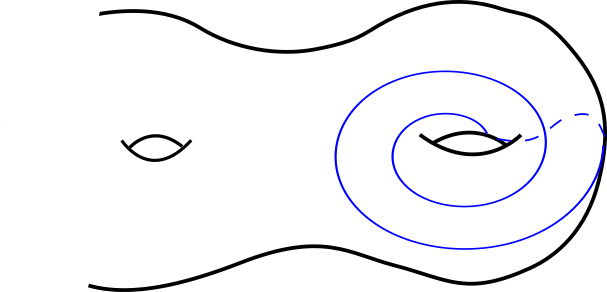}

\end{minipage}
	\caption{Handle attachment for $M$ (left) and $M\#\mathbb{RP}^3$ (right)}\label{fig:g1handle} 
	\end{figure}
	\begin{proposition}
	The map $i_*: \mathcal{S}_{\sqrt{-1}}(M)\ra\mathcal{S}_{\sqrt{-1}}(M\#\rp)$ is injective.
	\end{proposition}
	\begin{proof}
	Choose a Heegaard splitting of genus $g$ for $M$ and the Heegaard splitting for $\rp$ shown in Figure \ref{fig:g1handle} (left); call $\alpha_1,\dots\alpha_g$ a set of disjoint curves bounding disks for the Heegaard splitting of $M$ and $\alpha_{g+1}$ the curve in Figure \ref{fig:g1handle} (right). Then we can construct a Heegaard splitting of $M\#\rp$ by taking a genus $g+1$ handlebody  $H_{g+1}=H_g\natural H_1$ ($\natural$ is the boundary sum operation) and embedding in its boundary the curves $\alpha_1,\dots,\alpha_g$ (in the $H_g$ part) and $\alpha_{g+1}$ (in the $H_1$ part). Then by \ref{thm:handleslide} $\Sk(M\#\rp)=\Sk(H_{g+1})/I$, where $I$ is the ideal generated by handleslides. Specifically, $I$ is generated by the set $\{\gamma-sl_{\alpha_i}(\gamma), \gamma\subseteq M\#\rp, i=1,\dots,g+1\}$ where $\gamma$ is any curve in $M\#\rp$ and $sl_{\alpha_i}(\gamma)$ is its handleslide over $\alpha_i$.  
	As noted before, for this operation to be well defined, we need to choose an arc connecting $\gamma$ to $\alpha_i$; we do not write it explicitly in the notation but it is assumed that any such arc can be chosen. We decompose $I$ into three parts:
	$I_1\subseteq I$ the subset generated by handleslides of curves in $H_g$ over $\alpha_1,\dots,\alpha_g$; $I_2\subseteq I$ the subset generated by handleslides of curves in $H_{g+1}$ over $\alpha_1,\dots\alpha_g$ (notice that $I_1\subseteq I_2$); and $I_3\subseteq I$ the subset generated by handleslides of curves in $H_{g+1}$ over $\alpha_{g+1}$. Clearly $I=I_1+I_2+I_3$.
	
	Given an ideal $J\subseteq \Sk(M\#\rp)$, we denote with $\tilde{J}$ its image in $\Sk_{\sqrt{-1}}(M\#\rp)$.
	
	By choosing the Heegaard splitting like this, we have essentially also picked an embedding $j:H_g\ra H_{g+1}$ which induces a map $j_*:\Sk(H_g)\ra \Sk(H_{g+1})$; this is known to be an embedding, therefore from now on we identify $\Sk(H_g)$ with $j_*(\Sk(H_g))$. This inclusion clearly also induces an embedding $i: M\setminus B^3\ra M\#\rp$ which induces the map $i_*$ from the statement.
	
	Clearly $I_1\subseteq \Sk(H_g)$ and $\Sk(M)=\Sk(H_g)/I_1$; we now proceed to show that $\Sk_{\sqrt{-1}}(H_g)\cap \tilde{I}=\tilde{I_1}$ which concludes the proof.

	\end{proof}
	\begin{lemma}
	$I_2\cap \Sk(H_g)=I_1$
	\end{lemma}
	
	\begin{proof}
	Suppose there is $x\in I_2\cap \Sk(H_g)$; then because $x\in I_2$, we have $x=\sum_{j=1}^k \lambda_j \left(\gamma_j - sl_{\alpha_{i_j}}(\gamma_j)\right)$, with $1\leq i_j\leq g$. Consider now an embedding $\psi: H_{g+1}\ra H_g$ such that $\psi \circ j$ is isotopic to the identity; another way of thinking about $\psi$ is that it is obtained by "filling in" the last handle to obtain $H_g$. This induces a map $\Psi: \Sk(H_{g+1})\ra \Sk(H_g)$ which, by functoriality, must satisfy $\Psi\circ j_*=Id$. We claim that $\Psi\left(sl_{\alpha_{i}}(\gamma)\right)=sl_{\alpha_{i}}(\Psi(\gamma))$ for any $1\leq i\leq g$ and for any curve $\gamma\subseteq H_{g+1}$. This is because taking a curve in $H_{g+1}$, handlesliding it over a curve in $H_g$ and then filling in the $g+1$-th handle gives a curve that is equal, as a subset of $H_{g}$, to first filling in the $g+1$-th handle and subsequently handlesliding it over a curve in $H_g$.
	
On the one hand, if we assume that $x\in \Sk(H_g)$, we must have $\Psi(x)=x$. On the other hand, $\Psi(x)=\sum_{j=1}^k \lambda_j \Psi\left(\gamma_j - sl_{\alpha_{i_j}}(\gamma_j)\right)=\sum_{j=1}^k \lambda_j \left(\Psi(\gamma_j) - sl_{\alpha_{i_j}}(\Psi(\gamma_j))\right)$; because each $\Psi(\gamma_j)$ is in $H_g$, we have by definition that $x=\Psi(x)\in I_1$.
	\end{proof}	
	\begin{lemma}
	$\tilde{I_3}\cap \Sk_{\sqrt{-1}}(H_g)= \{0\}$
	\end{lemma}
\begin{figure}
\centering
\begin{minipage}{.19\textwidth}
	\centering    
	\includegraphics[width=0.9\textwidth]{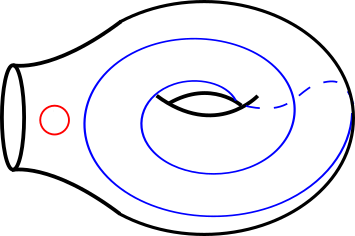}
	\label{fig:g1hs0before}
	\vspace{0.35cm}
\end{minipage}  $\longrightarrow$
\begin{minipage}{.19\textwidth}
	\centering    
	\includegraphics[width=0.9\textwidth]{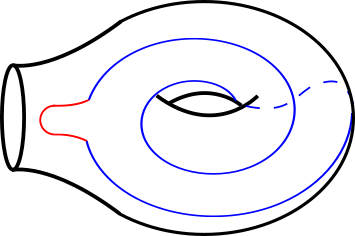}
	\label{fig:g1hs0after}
	\vspace{0.35cm}
\end{minipage}\hspace{1cm}
\begin{minipage}{.19\textwidth}
	\centering    
	\includegraphics[width=0.9\textwidth]{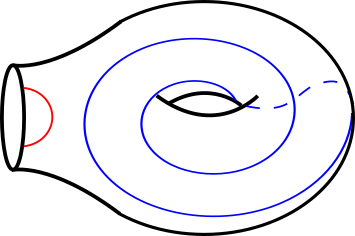}
	\label{fig:g1hs1before}
	\vspace{0.35cm}
\end{minipage}  $\longrightarrow$
\begin{minipage}{.19\textwidth}
	\centering    
	\includegraphics[width=0.9\textwidth]{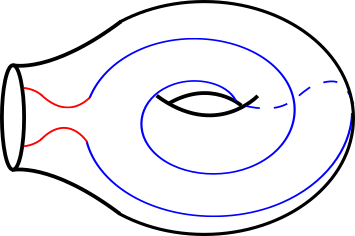}
	\label{fig:g1hs1after}
	\vspace{0.35cm}
\end{minipage}
\begin{minipage}{.19\textwidth}
	\centering    
	\includegraphics[width=0.9\textwidth]{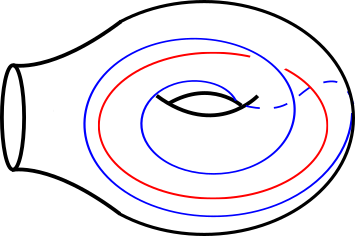}
	\label{fig:g1hs2before}
\end{minipage}  $\longrightarrow$
\begin{minipage}{.19\textwidth}
	\centering    
	\includegraphics[width=0.9\textwidth]{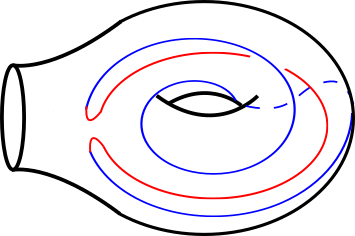}
	\label{fig:g1hs2after}
\end{minipage}
\hspace{1cm}
\begin{minipage}{.19\textwidth}
	\centering    
	\includegraphics[width=0.9\textwidth]{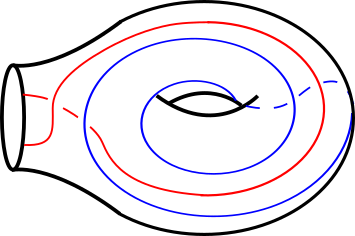}
	\label{fig:g1hs3before}
\end{minipage}  $\longrightarrow$
\begin{minipage}{.19\textwidth}
	\centering    
	\includegraphics[width=0.9\textwidth]{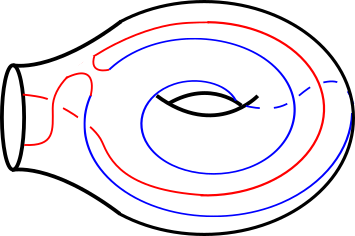}
	\label{fig:g1hs3after}
\end{minipage}
\centering
\begin{minipage}{.19\textwidth}
	\centering    
	\vspace{0.35cm}
	\includegraphics[width=0.9\textwidth]{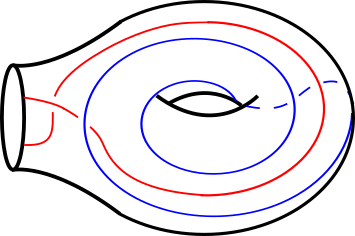}
	\label{fig:g1hs4before}
\end{minipage}  $\longrightarrow$
\begin{minipage}{.19\textwidth}
	\centering    
	\vspace{0.35cm}
	\includegraphics[width=0.9\textwidth]{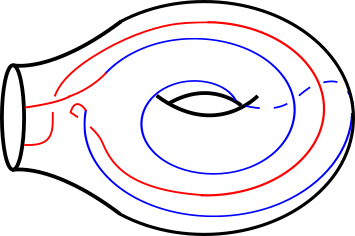}
	\label{fig:g1hs4after}
\end{minipage}
\caption{Handleslides over $\alpha_{g+1}$}\label{fig:g1handleslides}
\end{figure}
\begin{proof}
First we show that $I_3$ can be generated by handleslides of the forms shown in Figure \ref{fig:g1handleslides}, where we only show the portion of the link in $H_{g+1}$ that intersects the last handle and that is getting handle-slid (i.e. the link could have other components, and it could contain other arcs in the last handle).
To show this, take a multicurve $\Gamma$ and a curve $\gamma\subseteq \Gamma$ getting handle-slid over $\alpha_{g+1}$. Isotope $\gamma$ to be close to $\alpha$; then the handleslide replaces the red arc on the left of Figure \ref{fig:handleslide} with the red arc on the right. Now "anchor" the end points of the red arc to the boundary of $H_{g+1}$ and consider the complement of the red arc. We show that, using skein relations, isotopies that do not move the anchored endpoints and adding extra components, we can change the portion of $\gamma$ that intersects the anchored endpoints and the last handle to a linear combination of one of the cases of Figure \ref{fig:d1rel}. There are actually only 4 more planar ways that the arc(s) could be configured, shown in Figure \ref{fig:d1relextra}; we show how to deal with one case and the others are done in a similar fashion.

\begin{figure}
	\centering
	\begin{minipage}{.3\textwidth}
		\centering    
		\includegraphics[width=0.6\textwidth]{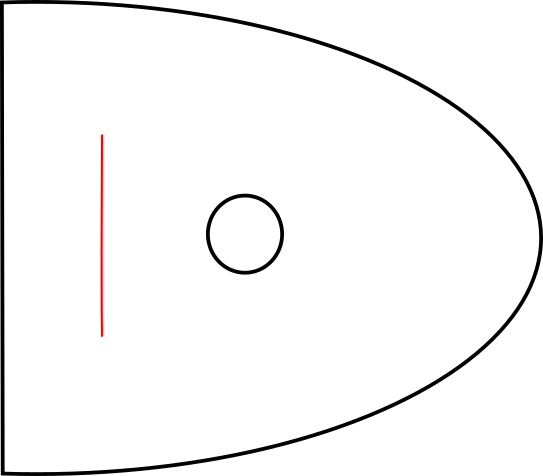}
	\end{minipage} 
	\begin{minipage}{.3\textwidth}
		\centering    
		\includegraphics[width=0.6\textwidth]{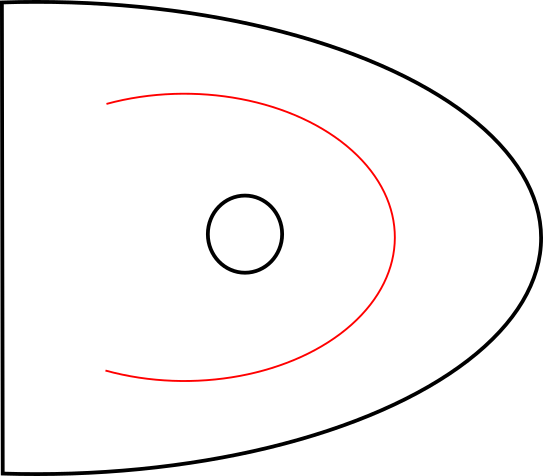}
	\end{minipage}
	
	\centering
	\begin{minipage}{.3\textwidth}
		\centering    
		\includegraphics[width=0.6\textwidth]{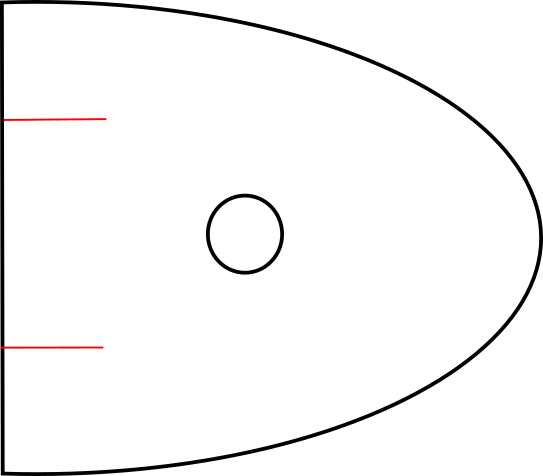}
	\end{minipage}  
	\begin{minipage}{.3\textwidth}
		\centering    
		\includegraphics[width=0.6\textwidth]{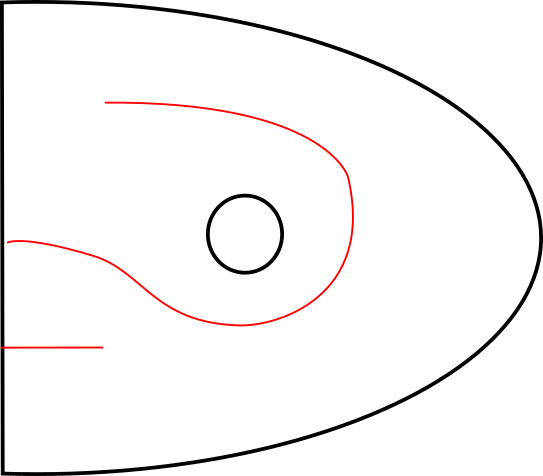}
	\end{minipage}
\begin{minipage}{.3\textwidth}
\centering    
\includegraphics[width=0.6\textwidth]{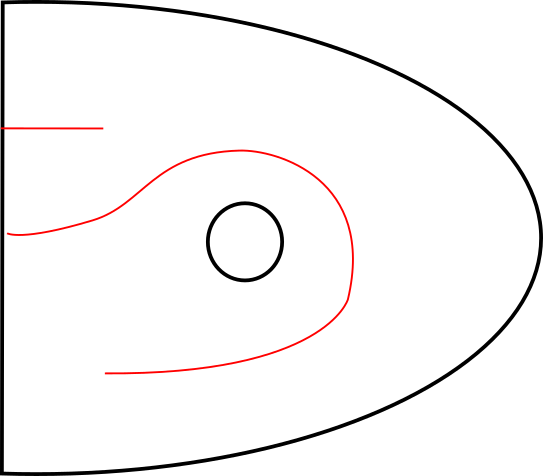}
\end{minipage}
	\caption{}\label{fig:d1rel}
\end{figure}

\begin{figure}
	\centering
	\begin{minipage}{.22\textwidth}
		\centering    
		\includegraphics[width=0.8\textwidth]{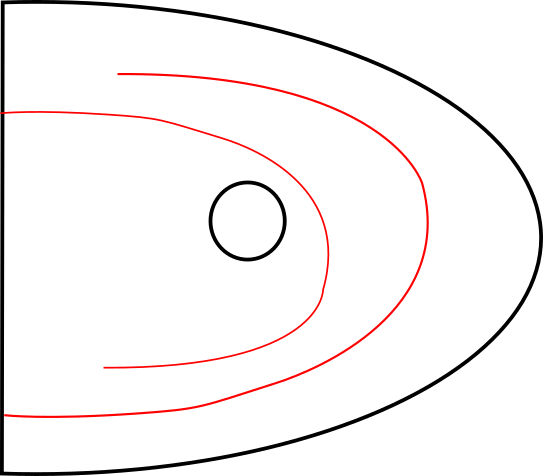}
	\end{minipage} 
	\begin{minipage}{.22\textwidth}
		\centering    
		\includegraphics[width=0.8\textwidth]{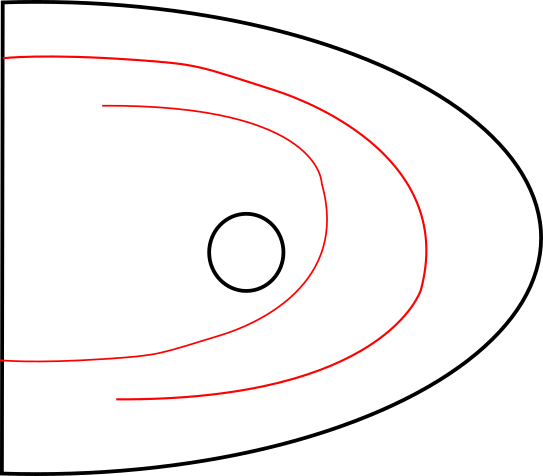}
	\end{minipage}
		\centering
	\begin{minipage}{.22\textwidth}
		\centering    
		\includegraphics[width=0.8\textwidth]{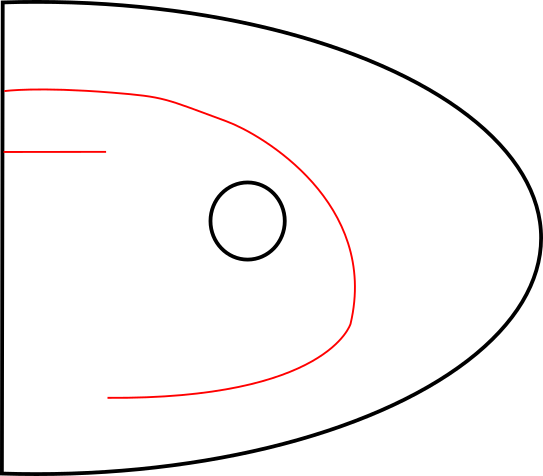}
	\end{minipage}  
	\begin{minipage}{.22\textwidth}
		\centering    
		\includegraphics[width=0.8\textwidth]{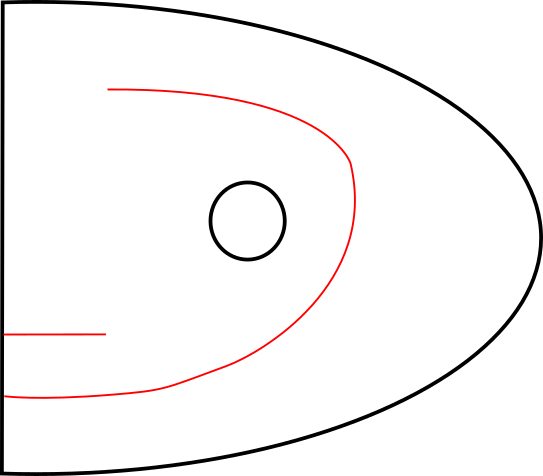}
	\end{minipage}
	\caption{}\label{fig:d1relextra}
\end{figure}
Consider for example the leftmost diagram of Figure \ref{fig:d1relextra}:
		\begin{align*}
	\vcenter{\hbox{\includegraphics[width=0.19\textwidth]{figures/d1relextra1}}}&=\vcenter{\hbox{\includegraphics[width=0.19\textwidth]{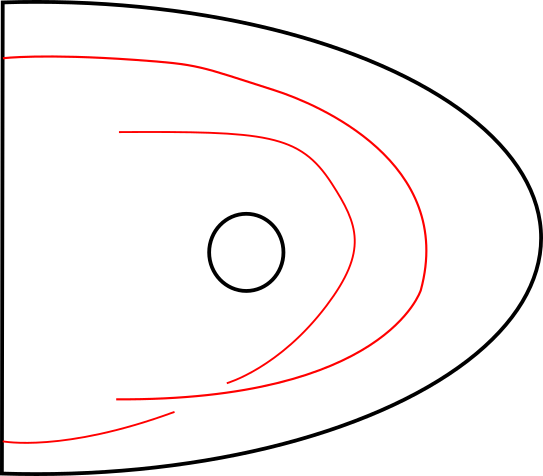}}}=A\hspace*{1mm}\vcenter{\hbox{\includegraphics[width=0.19\textwidth]{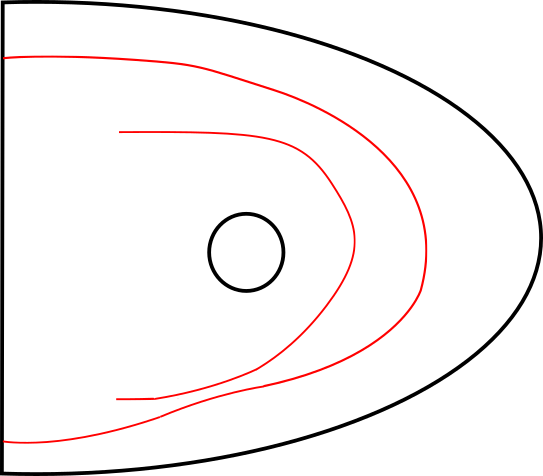}}}+A^{-1}\hspace*{1mm}\vcenter{\hbox{\includegraphics[width=0.19\textwidth]{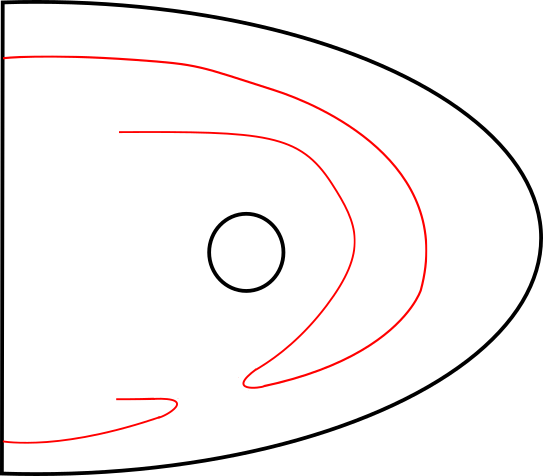}
	}}
\end{align*}
and the latter is a linear combination of elements from Figure \ref{fig:d1rel} with possibly some extra arcs.

Now that we have generators for $I_3$, we can pass to $\tilde{I_3}$ and write down explicitly the handle-slide relations; this follows a similar process as the one used to prove Proposition \ref{prop:presentationlp1l21}. Recall that in $\Sk_{\sqrt{-1}}(H_{g+1})$, a crossing change is the same as a sign change; therefore in all these calculations we push the handle-slid curve to the top and introduce a sign. Throughout the rest of the proof, $k(\Gamma)$ will be the algebraic intersection number of a multicurve $\Gamma\subseteq H_{g+1}$ with a compression disk for the $g+1$-th handle, and a notation of the form $\Gamma D$, with $D$ a planar diagram, denotes the multicurve obtained from $\Gamma$ by adding the curve depicted in $D$ close to $\partial H_{g+1}$.

Consider for example the bottom handleslide of Figure \ref{fig:g1handleslides}, where we are handlesliding a component $\gamma$ (in red) of a planar multicurve $\Gamma\sqcup\gamma$ (there could be other components of the multicurve, or indeed other arcs of the handleslid component, in the last handle).

The left hand side is equal to $\Gamma\left(A\hspace*{1mm}\vcenter{\hbox{\includegraphics[width=0.2\textwidth]{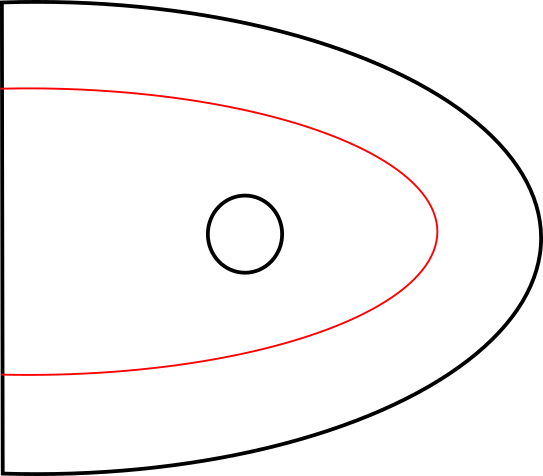}}}+A^{-1}\hspace*{1mm}\vcenter{\hbox{\includegraphics[width=0.2\textwidth]{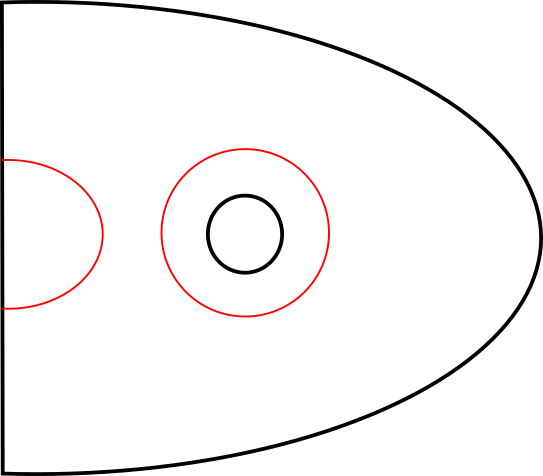}}}\right)$
whereas the right hand side is
$$(-1)^{k(\Gamma)}\Gamma\left(\hspace*{1mm}\vcenter{\hbox{\includegraphics[width=0.2\textwidth]{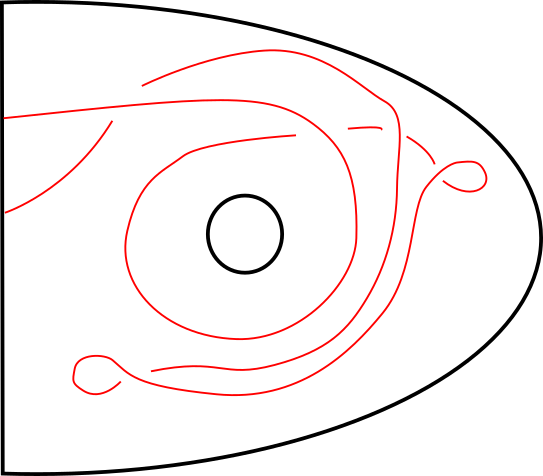}}}\right)=(-1)^{k(\Gamma)+1}A^{-3}\Gamma\left(\hspace*{1mm}\vcenter{\hbox{\includegraphics[width=0.2\textwidth]{figures/d1lhs2}}}\right)$$
after resolving the framings and performing an isotopy. Considering now that $A=\sqrt{-1}$ we have the relation
$$\sqrt{-1}(1+(-1)^{k(\Gamma)})\Gamma\left(\hspace*{1mm}\vcenter{\hbox{\includegraphics[width=0.2\textwidth]{figures/d1lhs2}}}\right)=\sqrt{-1}\Gamma\left(\hspace*{1mm}\vcenter{\hbox{\includegraphics[width=0.2\textwidth]{figures/d1lhs1}}}\right).
$$
We list the other handleslide relations without carrying out the calculations, since they are essentially identical to the ones seen so far.

The first relation of Figure \ref{fig:g1handleslides}  reads:

$$2\Gamma= (-1)^{k(\Gamma)+1}\Gamma\left(\hspace*{1mm}\vcenter{\hbox{\includegraphics[width=0.2\textwidth]{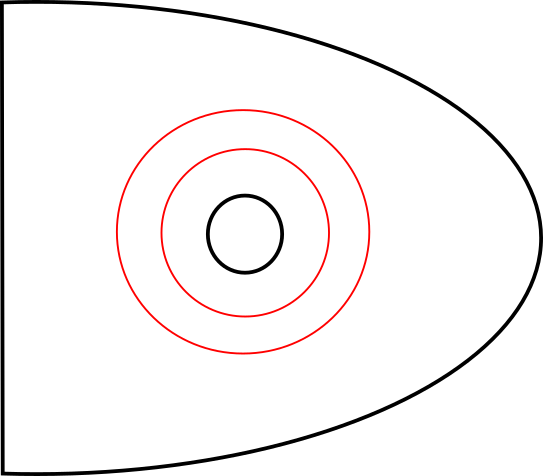}}}-2\right).$$
The second relation of Figure \ref{fig:g1handleslides}  reads:

$$\Gamma\left(\hspace*{1mm}\vcenter{\hbox{\includegraphics[width=0.2\textwidth]{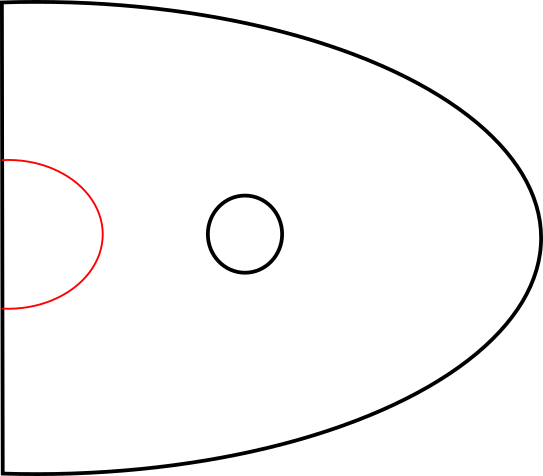}}}\right)= (-1)^{k(\Gamma)+1}\Gamma\left(\hspace*{1mm}\vcenter{\hbox{\includegraphics[width=0.2\textwidth]{figures/d1hs-6}}}-\vcenter{\hbox{\includegraphics[width=0.2\textwidth]{figures/d1lhs4}}}\right).$$
The third relation of Figure \ref{fig:g1handleslides}  reads:

$$\Gamma\left(\hspace*{0.8mm}\vcenter{\hbox{\includegraphics[width=0.16\textwidth]{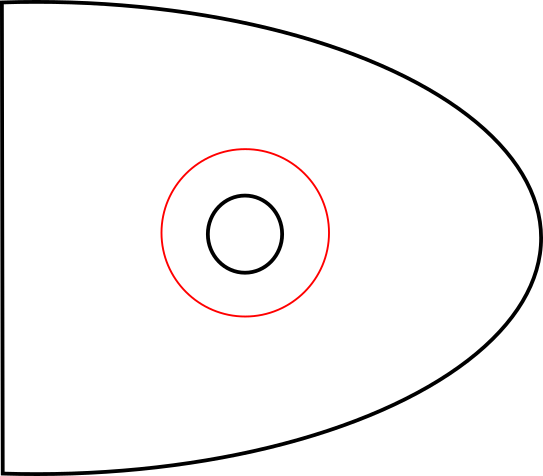}}}\right)=(-1)^{k(\Gamma)}\Gamma\left(\hspace*{0.8mm}\vcenter{\hbox{\includegraphics[width=0.16\textwidth]{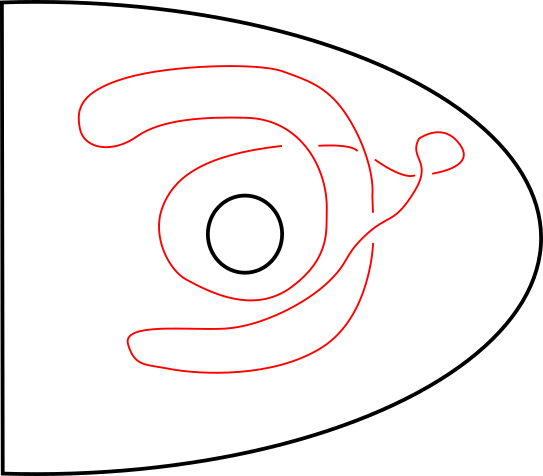}}}\right)= (-1)^{k(\Gamma)}\Gamma\left(\hspace*{0.8mm}\vcenter{\hbox{\includegraphics[width=0.16\textwidth]{figures/d1y}}}\right).$$

Finally, the fourth relation of Figure \ref{fig:g1handleslides}  reads:
$$\sqrt{-1}\Gamma\left(\hspace*{1mm}\vcenter{\hbox{\includegraphics[width=0.16\textwidth]{figures/d1lhs1}}}-\vcenter{\hbox{\includegraphics[width=0.16\textwidth]{figures/d1lhs2}}}\right)=(-1)^{k(\Gamma)+1}\sqrt{-1}\Gamma\left(\hspace*{1mm}\vcenter{\hbox{\includegraphics[width=0.16\textwidth]{figures/d1lhs1}}}\right).$$

Notice that in all the above cases, when there is an even number of intersection points between $\gamma\sqcup \Gamma$ and the compression disk for the last handle, the relation simplifies to a relation of the form $\Gamma'=0$, where $\Gamma'$ is a planar multicurve intersecting the last compression disk at least once. On the other hand, when there is an odd number of intersection points, this applies to all multicurves in the relation.

Now consider a linear combination $x$ of the above elements that belongs to $\Sk_{\sqrt{-1}}(H_g)$; the statement of the Lemma is that it must be equal to $0$. Notice that if $x=\sum_i\lambda_i\Gamma_i$ is a linear combination of multicurves, $x\in \Sk_{\sqrt{-1}}(H_g)$ and each $\Gamma_i$ intersects the last compression disk an odd number of times, then $x=0$; this is because $\Sk_{\sqrt{-1}}(H_{g})$ is graded by $H_1(H_{g+1},\mathbb{Z}/2\mathbb{Z})$ and all the $\Gamma_i$s have a non-trivial last component in $H_1(H_{g+1},\Z/2\Z)$. Similarly, if $x=\sum_i\lambda_i\Gamma_i$ is a linear combination of multicurves, $x\in \Sk_{\sqrt{-1}}(H_g)$ and for $1\leqslant i\leqslant k$, $\Gamma_i$ intersects the last compression disk an odd number of times, then $\sum_{i=1}^k\lambda_i\Gamma_i=0$. Therefore, we can assume that each multicurve in $x$ intersects the last compression disk an even number of times; however as previously noted this means that $x$ is a linear combination of planar multicurves, each intersecting the last compression disk at least once. Now, because planar multicurves form a basis for $\Sk_{\sqrt{-1}}(H_g)$ and $\Sk_{\sqrt{-1}}(H_{g+1})$, this means that $x=0$.
\end{proof}

\begin{corollary}
	If $M$ is such that $\Sk_{\sqrt{-1}}(M)$ is infinite dimensional, then so is $M\# \mathbb{RP}^3$; in particular, $\Sk(N)$ has torsion for $N$ a connected sum of $L(p,1)$ with $p\geq 2$ even and any (positive) number of copies of $\mathbb{RP}^3$.
\end{corollary}

	\bibliographystyle{hamsalpha}
	\bibliography{biblio}
\end{document}